%% file: VMS-POD-NS-R1.tex
\documentclass[11pt]{article}
\usepackage{latexsym,amsmath,amssymb,epsfig,bbm,bm,graphicx,xcolor,cancel,mathrsfs}
\usepackage{textcomp}
\usepackage{multirow}
\usepackage{caption}  
\usepackage[left]{lineno}

\usepackage{epstopdf}
\usepackage{authblk}
\usepackage{setspace}
\newcommand{\qed}{\hfill \ensuremath{\Box}}

\newcommand{\vertiii}[1]{{|\!|\!| #1 |\!|\!|}}

\makeatletter
      \newcommand\figcaption{\def\@captype{figure}\caption}
      \newcommand\tabcaption{\def\@captype{table}\caption}
\makeatother
\newcommand{\me}{\mathcal{E}}

\graphicspath{{./Figures/}}

\title{Variational Multiscale Proper Orthogonal Decomposition: Navier-Stokes Equations}
\author[1]{Traian Iliescu}
\author[2]{Zhu Wang\thanks{Corresponding author (wangzhu@ima.umn.edu).}}
\affil[1]{Department of Mathematics, Virginia Polytechnic Institute 
and State University, Blacksburg, VA 24061-0123, U.S.A.}
\affil[2]{Institute for Mathematics and its Applications, University of Minnesota, Minneapolis, MN 55455-0134, U.S.A.}

\date{Revision 1 -- \today}

\begin{document}
\linenumbers
\doublespacing
\maketitle
\input{notation_old}

\begin{abstract}
We develop a variational multiscale proper orthogonal decomposition reduced-order model for turbulent incompressible Navier-Stokes equations. 
The error analysis of the full discretization of the model is presented.
All error contributions are considered: 
the spatial discretization error (due to the finite element discretization),
the temporal discretization error (due to the backward Euler method), and
the proper orthogonal decomposition truncation error.
Numerical tests for a three-dimensional turbulent flow past a cylinder at Reynolds number $\mbox{Re}=1000$ 
show the improved physical accuracy of the new model over the standard Galerkin and mixing-length proper orthogonal decomposition reduced-order models. 
The high computational efficiency of the new model is also showcased.
Finally, the theoretical error estimates are confirmed by numerical simulations of a two-dimensional Navier-Stokes problem. 
\end{abstract}

\input{Introduction}

\input{Model}

\input{Estimate}
\input{Numerics}
\input{Conclusion}

\section*{Acknowledgement}
The first author acknowledges the partial support by the National Science Foundation (DMS-0513542 and OCE-0620464). The second author was supported by the Institute for Mathematics and its Applications with funds provided by the National Science Foundation. We thank the anonymous reviewers for their constructive comments, which helped improve the manuscript.

\bibliographystyle{plain}
{\small
\bibliography{D_bibliography}
}
\end{document}

%% file: notation_old.tex
\newcommand{\lp}{\left(}
\newcommand{\rp}{\right)}
\newcommand{\lno}{\left\|}
\newcommand{\rno}{\right\|}

\newcommand{\ou}{\overline{u}}
\newcommand{\opsi}{\overline{\psi}}
\newcommand{\oq}{\overline{q}}
\newcommand{\oT}{\overline{T}}
\newcommand{\E}{\mathbbm{E}}
\newcommand{\orho}{\overline{\rho}}

\newcommand{\s}{\sigma}
\renewcommand{\k}{\kappa}
\newcommand{\p}{\partial}
\newcommand{\om}{\omega}
\newcommand{\Om}{\Omega}
\newcommand{\pOm}{\partial \Omega}
\newcommand{\e}{\epsilon}
\renewcommand{\a}{\alpha}
\renewcommand{\b}{\beta}
   \newcommand{\eps}{\varepsilon}
   \newcommand{\EX}{{\Bbb{E}}}
   \newcommand{\PX}{{\Bbb{P}}}

\newcommand{\nd}{{\nabla \cdot}}

\newcommand{\cF}{{\cal F}}
 \newcommand{\cH}{{\cal H}}
\newcommand{\cD}{{\cal D}}
\newcommand{\cO}{{\cal O}}
\newcommand{\cP}{\mathcal{P}}

\newtheorem{remark}{Remark}[section]
\newtheorem{lemma}{Lemma}[section]
\newtheorem{theorem}{Theorem}[section]
\newtheorem{corollary}{Corollary}[section]
\newtheorem{proposition}{Proposition}[section]
\newtheorem{definition}{Definition}[section]
\newtheorem{assumption}{Assumption}[section]
\newenvironment{proof}[1][Proof]{\begin{trivlist}
\item[\hskip \labelsep {\bfseries #1}]}{\end{trivlist}}
\newcommand{\kmodel}{k_{\mbox{Model}}}
\newcommand{\obu}{\overline{\bf u}}
\newcommand{\oobu}{\overline{\overline{\bf u}}}
\newcommand{\bu}{{\bf u}}
\newcommand{\bk}{{\bf k}}
\newcommand{\bs}{{\bf s}}
\newcommand{\oou}{\overline{\overline{u}}}
\newcommand{\op}{\overline{p}}
\newcommand{\of}{\overline{f}}
\newcommand{\obf}{\overline{\bf f}}
\newcommand{\ow}{\overline{w}}
\newcommand{\ov}{\overline{v}}
\newcommand{\ophi}{\overline{\phi}}
\newcommand{\oS}{\overline{S}}
\newcommand{\obS}{\overline{\bf S}}
\newcommand{\bv}{{\bf v}}
\newcommand{\obv}{\overline{\bf v}}
\newcommand{\bc}{{\bf c}}
\newcommand{\by}{{\bf y}}
\newcommand{\bw}{{\bf w}}
\newcommand{\bW}{{\bf W}}
\newcommand{\bU}{{\bf U}}
\newcommand{\obw}{\overline{\bf w}}
\newcommand{\bz}{{\bf z}}
\newcommand{\bZ}{{\bf Z}}
\newcommand{\obZ}{\overline{\bf Z}}
\newcommand{\bff}{{\bf f}}
\newcommand{\bee}{{\bf e}}
\newcommand{\bn}{{\bf n}}
\newcommand{\bx}{{\bf x}}
\newcommand{\bX}{{\bf X}}
\newcommand{\bH}{{\bf H}}
\newcommand{\bV}{{\bf V}}
\newcommand{\bL}{{\bf L}}
\newcommand{\bg}{{\bf g}}
\newcommand{\bj}{{\bf j}}
\newcommand{\br}{{\bf r}}
\newcommand{\grads}{\nabla^s}
\def\PP{{{\rm l}\kern - .15em {\rm P} }}
\def\PN2{{\PP_{N}-\PP_{N-2}}}
\newcommand{\erf}[1]{\mbox{erf}\left(#1\right)}
\newcommand{\D}{\nabla}
\newcommand{\id}{\mathrm{d}}
\newcommand{\I}{\mathbb{I}}
\newcommand{\N}{\mathbb{N}}
\newcommand{\R}{\mathbb{R}}
\newcommand{\Z}{\mathbb{Z}}
\newcommand{\mathR}{\R}
\newcommand{\mathN}{\N}
\newcommand{\mathZ}{\Z}
\newcommand{\mathI}{\mathbbm{I}}
\newcommand{\btau}{\boldsymbol{\tau}}
\newcommand{\bphi}{\boldsymbol{\varphi}}
\newcommand{\bvarphi}{\boldsymbol{\varphi}}
\newcommand{\bpsi}{\boldsymbol{\psi}}
\newcommand{\bfeta}{\boldsymbol{\eta}}
\newcommand{\blambda}{\boldsymbol{\lambda}}
\newcommand{\bPhi}{\boldsymbol{\Phi}}
\newcommand{\obphi}{\overline {\boldsymbol{\phi}}}
\newcommand{\bomega}{\boldsymbol{\omega}}
\newcommand{\bsigma}{\boldsymbol{\sigma}}
\newcommand{\orhoprime}{\overline{\rho^{\prime}}}
\newcommand{\bus}{{\bf u}^*}
\newcommand{\By}{\mathcal B(\by)}
\newcommand{\eci}[1]{\mathcal E_{#1}}
\newcommand{\dpyi}[1]{\delta_{#1}^+(\by)}
\newcommand{\dmyi}[1]{\delta_{#1}^-(\by)}
\newcommand{\cA}{{\mathcal A(\by)}}
\newcommand{\dyi}[1]{\delta_{#1}(\by)}
\newcommand{\cG}{{\mathcal G(\bx,\by)}}
\newcommand{\cGi}[1]{{\mathcal G_{#1}(\bx,\by)}}
\newcommand{\pti}{\partial_i}
\newcommand{\ptii}[1]{\partial_{#1}}
\newcommand{\ba}{{\bf a}}

\newcommand{\rey}{\mbox{Re}}

\newcommand{\tnp}{t^{k+1}}
\newcommand{\bb}{{\bf b}}
\newcommand{\fnp}{f^{k+1}}
\newcommand{\prp}{P_R^{'}}
\newcommand{\pr}{P_R}
\newcommand{\rn}{{\bf r}^k}
\newcommand{\ur}{\bu_r}
\newcommand{\urn}{\bu_r^k}
\newcommand{\utn}{\bu_t^k}
\newcommand{\urnp}{\bu_r^{k+1}}
\newcommand{\utnp}{\bu_t^{k+1}}
\newcommand{\un}{\bu^k}
\newcommand{\unp}{\bu^{k+1}}
\newcommand{\vr}{\bv_r}
\newcommand{\wrn}{\bw_r^k}
\newcommand{\wrnp}{\bw_r^{k+1}}
\newcommand{\wRnp}{\bw_R^{k+1}}
\newcommand{\etan}{{\boldsymbol \eta}^k}
\newcommand{\etanp}{{\boldsymbol \eta}^{k+1}}
\newcommand{\prn}{{\boldsymbol \phi}_r^k}
\newcommand{\prnp}{{\boldsymbol \phi}_r^{k+1}}
\newcommand{\prz}{{\boldsymbol \phi}_r^{0}}
\newcommand{\prm}{{\boldsymbol \phi}_{r}^{M}}
\newcommand{\sNn}{\sum\limits_{k=0}^{M-1}}
\newcommand{\sN}{\sum\limits_{k=0}^{M}}
\newcommand{\asNn}{\frac{1}{M}\sNn}
\newcommand{\asN}{\frac{1}{M+1}\sN}
\newcommand{\uph}{\upsilon^h}
\newcommand{\half}{\frac{1}{2}}

\newcommand{\opartial}{\overline{\partial}}

%% file: Introduction.tex
\section{Introduction}

Due to the complexity of fluid flows in many realistic engineering problems, millions or even billions of degrees of freedom are often required in a direct numerical simulation (DNS).  
To allow efficient numerical simulations in these applications, reduced-order models (ROMs) are often used.
The proper orthogonal decomposition (POD) has been one of the most popular approaches employed in developing ROMs for complex fluid flows \cite{AK04,lumley1967sit,Sir87a,Sir87b,Sir87c}. 
%
It starts by using a DNS (or experimental data) to generate a POD basis $\{ \bphi_{1}, \ldots, \bphi_d \}$ that maximizes the energy content in the system, where $d$ is the rank of the data set.
By utilizing the Galerkin method, one can project the original system onto the space spanned by only a handful of  dominant POD basis functions $\{ \bphi_{1}, \ldots, \bphi_r \}$, with $r \leq d$, which results in a low-order model ---  the Galerkin projection-based POD-ROM (POD-G-ROM). 

The POD-G-ROM has been applied successfully in the numerical simulation of laminar flows. 
It is well known, however, that a simple POD-G-ROM will generally produce erroneous results for turbulent flows \cite{Aubry1988}.
The reason is that, although the discarded POD modes $\{ \bphi_{r+1}, \ldots, \bphi_d \}$ only contain a small part of the system's kinetic energy, they do, however, have a significant impact on the dynamics.
To model the effect of the discarded POD modes, various approaches have been proposed (see, e.g., the survey in \cite{wang2011closure}).
In this report, we develop an approach that improves the physical accuracy of the POD-ROM for turbulent incompressible fluid flows by utilizing a variational multiscale (VMS) idea \cite{HMOW01,HOM01}. This method is an extension to the Navier-Stokes equations (NSE) of the VMS-POD-ROM that we proposed in \cite{wang2011variational} for convection-dominated convection-diffusion-reaction equations. Our approach employs an eddy viscosity (EV) to model the interaction between  the discarded POD modes and those retained in the POD-ROM. 
%
Instead of being added to all the resolved POD modes $\{ \bphi_{1}, \ldots, \bphi_r \}$, EV is only added to the small resolved scales (POD modes $\{ \bphi_{R+1}, \ldots, \bphi_r \}$ with $R<r$) in the VMS-POD-ROM.
Thus, the small scale oscillations are eliminated without polluting the large scale 
components of the approximation. 
The small scales in the VMS-POD-ROM are defined by a projection approach 
in \cite{wang2011variational}, which is also used in \cite{john2005finite,john2008finite_a,john2008finite_b,john2006two,layton2002connection} in the finite element (FE) context. 
We also note that a different approach was developed in  \cite{guermond1999stabilization,guermond1999stabilisation}.

In this report, the VMS-POD-ROM is extended and studied for the NSE.
A rigorous error analysis of the full discretization of the model (FE in space, backward Euler in time) is presented. 
A numerical test of the VMS-POD-ROM for three-dimensional (3D) turbulent flow past a circular cylinder at Reynolds number $\rey = 1000$ is conducted to investigate the physical accuracy of the model.
The theoretical error estimates are confirmed by using the VMS-POD-ROM in the numerical simulation of a two-dimensional (2D) flow.

The rest of this paper is organized as follows: In Section \ref{mod}, we briefly describe the POD methodology and introduce the new VMS-POD-ROM. The error analysis for the full discretization of the new model is presented in Section \ref{err}. The new methodology is tested numerically in Section \ref{num} for 
a 3D flow past a circular cylinder and 
a 2D flow problem. Finally, Section \ref{con} presents the conclusions and future research directions.

%% file: Model.tex
\section{Variational Multiscale Proper Orthogonal Decomposition}\label{mod}
We consider the numerical solution of the incompressible Navier-Stokes equations:
\begin{equation}
\label{eq:nse}
\left\{ 
\begin{array}{cc}
\bu_{t} - \nu \Delta \bu + \lp \bu \cdot \nabla\rp  \bu + \nabla p = \bff , & \text{ in } \Omega \times (0,T], \\
\nabla \cdot \bu = 0 , & \text{ in } \Omega \times (0,T] , \\
\bu = 0, & \text{ on } \partial \Omega \times (0,T] , \\
\bu({\bx}, 0) = \bu^0(\bx), & \text{ in }  \Omega ,
\end{array} 
\right.
\end{equation}
where $\bu(\bx, t)$ and $p(\bx, t)$ represent the fluid velocity and pressure of a flow in the region $\Om$, respectively, for $\bx\in \Om$, $t\in [0, T]$, and $\Om \subset \mathbb{R}^n$ with $n= 2$ or $3$; 
the flow is bounded by walls and driven by the force $\bff(\bx, t)$;  
$\nu$ is the reciprocal of the Reynolds number; 
and $\bu^0(\bx)$ denotes the initial velocity. 
We also assume that the boundary of the domain, $\pOm$, is polygonal when $n= 2$ and is polyhedral when $n= 3$.

The following functional spaces and notations will be used in the paper:
\begin{equation*}
	\bX=\bH^1_0\lp\Om\rp =\left\{\bv\in [L^2(\Om)]^n: \nabla \bv \in [L^2(\Om)]^{n\times n} \text{ and } \bv= {\bf 0} \text{ on } \p \Om \right\},
\end{equation*}	
\begin{equation*}
	Q=L^2_0\lp\Om\rp = \left\{q\in L^2(\Om): \int_{\Om} q\, \id {\bx} = 0 \right\}, 
\end{equation*}	
\begin{equation*}	
	\bV=\left\{\bv\in \bX: \lp \nabla \cdot \bv, q\rp  = 0, \forall\, q \in Q \right\}, \text{ and  }
\end{equation*}	
\begin{equation*}	
	\bV^h=\left\{\bv_h \in \bX^h: \lp \nabla \cdot \bv_h, q_h \rp  = 0, \forall\, q_h \in Q^h \right\}, 
\end{equation*}
where $\bX^h \subset \bX$ and $Q^h\subset Q$ are the FE spaces of the velocity and pressure, respectively.
In what follows, we consider the div-stable pair of FE spaces $\PP^{m} / \PP^{m-1}, \, m \geq 2$ \cite{layton2008introduction}. 
That is, the FE approximation of the velocity is continuous on $\Om$ and is an $n$-vector valued function with each component a polynomial of degree less than or equal to $m$ when restricted to an element, while that of the pressure is also continuous on $\Om$ and is a single valued function that is a polynomial of degree less than or equal to $m-1$ when restricted to an element. 
We emphasize, however, that our analysis extends to more general FE spaces. 
We consider the following spaces for the POD setting: 
\begin{subequations}
\begin{eqnarray}
\bX^r &:=& \text{span}\left\{\bphi_1, \bphi_2, \ldots, \bphi_r\right\}, \\
\bX^R &:=& \text{span} \left\{ \bphi_1, \bphi_2, \ldots, \bphi_R \right\}, \text{ and  } \\
\bL^R &:=& \text{span}\left\{ \nabla\bphi_1, \nabla\bphi_2, \ldots, \nabla\bphi_R \right\}, \label{eq:LR}
\end{eqnarray}
\end{subequations}
where $\bphi_j$, $j= 1, \ldots, r$, are the POD basis functions that will be defined in Section \ref{sec:pod}.  
We note that $\bX^R \subset \bX^r$, since $R< r$. 

We introduce the following notations: 
Let $\cal H$ be a real Hilbert space endowed with inner product $(\cdot, \cdot)_{\cal H}$ and norm $\|\cdot\|_{\cal H}$. 
Let the trilinear form $b^*(\cdot, \cdot, \cdot)$ be defined as 
\linelabel{line:b*}

\begin{equation*}
b^*\lp \bu, \bv, \bw\rp  =\frac{1}{2}\left[ \lp \lp \bu\cdot\nabla \rp \bv, \bw\rp - \lp \lp \bu\cdot\nabla \rp \bw, \bv\rp \right]
\label{eq:b*}
\end{equation*}
and the norm $\vertiii{\cdot}$ be defined as $\vertiii{\bf v}_{s, k}
:= \left(
	\frac{1}{M} \sum \limits_{i= 0}^{M-1} \| {\bf v}(\cdot, t_{i+1}) \|_k^{s}
\right)^{1/s}$, 
\linelabel{line:newnorm} 
where $s$ and $M$ are positive integers.

The weak formulation of the NSE \eqref{eq:nse} reads: 
Find $\bu \in \bX$ and $p \in Q$ such that
\begin{equation}
\left\{ 
	\begin{array}{cc}
		\lp  \bu_t , \bv \rp 
		+  \nu \lp  \D\bu , \D\bv \rp 
		+ b^*\lp \bu,\bu,\bv\rp 
		- \lp p , \nabla \cdot \bv \rp
		= \lp \bff, \bv\rp , 
		& \quad \forall \, \bv \in \bX, \\
		\lp \nabla \cdot \bu , q \rp 
		= 0,
		& \quad \forall \, q \in Q .
	\end{array} 
\right.
\label{eq:nse_weak}
\end{equation}
To ensure the uniqueness of the solution to \eqref{eq:nse_weak}, we make the following regularity assumptions (see Definition 29 and Remark 10 in \cite{layton2008introduction}): 
\begin{assumption}
\label{assumption_regularity}
In \eqref{eq:nse}, assume that 
$\bff\in L^2\lp 0, T; \bL^2(\Om)\rp$, 
$\bu^0\in \bV$, 
$\bu\in L^2\lp 0, T; \bX \rp \bigcap L^{\infty} \lp 0, T; \bL^2(\Om) \rp$, 
$\nabla \bu\in\lp L^4\lp 0, T; L^2(\Om)\rp \rp ^{n\times n}$, 
$\bu_t\in L^2\lp 0, T; \bX^*\rp$, and 
$p\in L^2 \lp 0, T; Q \rp$. 
\end{assumption}

The FE approximation of \eqref{eq:nse_weak} can be written as follows:
Find $\bu_h \in \bV^h$ such that
\begin{equation}
	\lp  \bu_{h,t} , \bv_h \rp 
	+  \nu \lp  \D\bu_h , \D\bv_h \rp 
	+ b^*\lp \bu_h,\bu_h,\bv_h \rp
	= \lp \bff, \bv_h \rp , 
	\quad \forall \, \bv_h \in \bV^h
\label{eq:nse_fe}
\end{equation}
and 
$\bu_h(\cdot, 0) = \bu^0_h\in \bV^h$.


\subsection{Proper Orthogonal Decomposition}\label{sec:pod}
We briefly describe the POD method, following \cite{KV01}.
For a detailed presentation, the reader is referred to \cite{chapelle2012galerkin,HLB96,singler2012new,Sir87a,volkwein2011model}.

Consider an ensemble of snapshots
$ \mathcal{R} := \mbox{span}\left\{ \bu(\cdot, t_0), \ldots, \bu(\cdot, t_M) \right\}$, which is a collection of velocity data from either numerical simulation results or experimental observations at time $t_i = i \Delta t$, $i=0, \ldots, M$ and let $\Delta t = \frac{T}{M}$.  
The POD method seeks a low-dimensional basis $\{ \bphi_1, \ldots, \bphi_r\}$ in $\cH$ 
that optimally approximates the snapshots in the following sense:
\begin{equation}
  \min \frac{1}{M+1} \sum_{\ell=0}^M 
  \left\| \bu(\cdot, t_{\ell}) - 
  \sum_{j=1}^r \left( \bu(\cdot, t_{\ell}),\bphi_j(\cdot) \right)_{\cH} \, \bphi_j(\cdot) 
  \right\|_{\cH}^2 
\label{pod_min}
\end{equation}
subject to the conditions that $(\bphi_j, \bphi_i)_{\cH} = \delta_{ij}, \ 1 \leq i, j \leq r$, where $\delta_{ij}$ is the Kronecker delta. 
To solve \eqref{pod_min}, one can consider the eigenvalue problem 
\begin{equation}
K \, \bz_j = \lambda_j \, \bz_j, \quad \text{for }j=1, \ldots, r,
\label{pod_eigenvalue}
\end{equation}
where $K \in \R^{(M+1) \times (M+1)}$ 
is the snapshot correlation matrix with entries  
\linelabel{line:snapshot_correlation}
$\displaystyle K_{k\ell} = \frac{1}{M+1} \left( \bu(\cdot, t_{\ell}) , \bu(\cdot, t_k) \right)_{\cH}$ 
for $\ell, k = 0, \ldots, M$,
$\bz_j$ is the $j$-th eigenvector, and 
$\lambda_j$ is the associated eigenvalue. 
The eigenvalues are positive and sorted in descending order $\lambda_1\geq \ldots \geq \lambda_r \geq 0$.  
It can then be shown that the solution
of \eqref{pod_min}, the POD basis function, is given by
\begin{equation}
\bphi_{j}(\cdot) = \frac{1}{\sqrt{\lambda_j}} \, \sum_{{\ell}= 0}^{M} (\bz_j)_{\ell} \, \bu(\cdot , t_{\ell}),
\quad 1 \leq j \leq r,
\label{pod_basis_formula}
\end{equation}
where $(\bz_j)_{\ell}$ is the $\ell$-th component of the eigenvector $\bz_j$. 
It can also be shown that the following error formula holds \cite{HLB96,KV01}:
\begin{equation}
\frac{1}{M+1} \sum_{{\ell}= 0}^M \left\| \bu(\cdot,t_{\ell}) - 
  \sum_{j=1}^r \left( \bu(\cdot,t_{\ell}),\bphi_j(\cdot) \right)_{\cH} \, \bphi_j(\cdot) 
  \right\|_{\cH}^2
= \sum_{j=r+1}^{d} \lambda_j \, ,
\label{pod_error_formula}
\end{equation}
where $d$ is the rank of $\mathcal{R}$. 

\begin{remark}
\label{rem:pod}
Since, as shown in \eqref{pod_basis_formula}, the POD basis functions are linear combinations of the snapshots , 
the POD basis functions satisfy the boundary conditions in \eqref{eq:nse} and are solenoidal.  
If the FE approximations are used as snapshots, the POD basis functions belong to $\bV^h$, which yields  $\bX^r\subset \bV^h$.
\end{remark}

The Galerkin projection-based POD-ROM employs both Galerkin truncation and Galerkin projection. The former yields an approximation of the velocity field by a linear combination of the truncated POD basis: 
\begin{equation}
\label{eq:pod}
  {\bu}\lp {\bf x}, t \rp \approx {\ur} \lp {\bf x}, t \rp 
  \equiv \sum_{j=1}^r a_j \lp t \rp \bphi_j \lp {\bf x} \rp,
\end{equation}
where 
$\left\{a_{j}\lp t\rp \right\}_{j=1}^{r}$ are the sought time-varying coefficients representing the POD-Galerkin trajectories. 
Note that $r\ll N$, where $N$ denotes the number of degrees of freedom in a DNS. 
Replacing the velocity $\bu$ with $\ur$ in the NSE \eqref{eq:nse}, using the Galerkin method, and projecting the resulted equations onto the space $\bX^r$, 
one obtains the {\it POD-G-ROM} for the NSE: Find $\ur\in\bX^r$ such that 
\begin{equation}
\lp  \frac{\partial \bu_{r}}{\partial t} , \bphi \rp 
+  \nu\lp  \D\ur , \D\bphi \rp 
+ b^*\lp \ur, \ur, \bphi\rp 
= \lp {\bf f}, \bphi\rp,  \quad
\forall \, \bphi \in \bX^r
\label{eq:pod-g}
\end{equation}
and $\bu_r(\cdot, 0) = \bu_r^0 \in \bX^r$. 
In \eqref{eq:pod-g}, the pressure term vanishes due to the fact that all POD modes are solenoidal and satisfy the appropriate boundary conditions. 
The spatial and temporal discretizations of \eqref{eq:pod-g} were considered in \cite{KV02,LCNY08}. 
Despite its appealing computational efficiency, the POD-G-ROM \eqref{eq:pod-g} has generally been limited to 
laminar flows.  
To overcome this restriction, we develop a closure method for the POD-ROM, which stems from the variational multiscale ideas.  

\subsection{Variational Multiscale Method}

Based on the concept of energy cascade and locality of energy transfer, the VMS method models the effect of unresolved scales by introducing extra eddy viscosities {\em to and only to} the resolved small scales \cite{HMOW01,HOM01}. 
For a standard FE discretization, the separation of scales is generally challenging. 
Indeed, unless special care is taken (e.g., mesh adaptivity is used), the FE basis does not include any a priori information regarding the scales displayed by the underlying problem. 
Since the POD basis functions are already listed in descending order of their kinetic energy content, the POD represents an ideal setting for the VMS methodology. 
Naturally, we regard the discarded POD basis functions as unresolved scales, $\{\bphi_1, \ldots, \bphi_R\}$ as resolved large scales, and $\{\bphi_{R+1}, \ldots, \bphi_{r}\}$ as resolved small scales, where $R < r$. 

We consider the orthogonal projection of   
$\bL^2$ on $\bL^R$, $\pr : \bL^2  \longrightarrow \bL^R$, defined by
\begin{equation}
\lp  \bu - \pr \bu, \bv_R \rp  = 0,
\quad \forall \, \bv_R \in \bL^R .
\label{def_pr}
\end{equation}
Let $\prp := \I - \pr$, where $\I$ is the identity operator.
We propose the {\em variational multiscale proper orthogonal decomposition reduced-order model ($P_R$-VMS-POD-ROM)} for the NSE: 
Find $\ur \in \bX^r$ such that
\begin{equation}
\lp  \frac{\partial \bu_{r}}{\partial t} , \bphi \rp 
+ \nu \lp  \D\bu_{r} , \D\bphi \rp 
+ b^*\lp \ur, \ur, \bphi\rp 
+ \alpha \, \lp \prp \nabla \ur , \prp \nabla \bphi\rp 
= \lp \bff, \bphi\rp,  \quad
\forall \, \bphi \in \bX^r,
\label{eq:vms-pod}
\end{equation}
where $\alpha>0$ is a constant EV coefficient and the initial condition is given by the $L^2$ projection of $\bu^0$ on $\bX^r$: 
\begin{equation}
\bu_r(\cdot, 0) = \bu_r^0 := \sum_{j=1}^r (\bu^0, \bphi_j)\bphi_j.
\label{eq:ic}
\end{equation} 

\begin{remark}
When $R=r$ or  $\alpha = 0$, the $P_R$-VMS-POD-ROM \eqref{eq:vms-pod} coincides with the standard POD-G-ROM, since no EV is introduced. 
When $R=0$, since EV is added to all modes in the POD-ROM, the $P_R$-VMS-POD-ROM \eqref{eq:vms-pod} becomes the {\it mixing-length POD-ROM (ML-POD-ROM)} \cite{Aubry1988,wang2011closure}: 
\begin{equation}
\lp  \frac{\partial \bu_{r}}{\partial t} , \bphi \rp 
+ \nu \lp  \D\bu_{r} , \D\bphi \rp 
+ b^*\lp \ur, \ur, \phi\rp 
+ \alpha \, \lp \nabla \ur, \nabla \bphi\rp 
= \lp \bff, \bphi\rp,  \quad
\forall \, \bphi \in \bX^r.
\label{eq:ml-pod}
\end{equation}
\end{remark}

\begin{remark}
\label{rem:vms_model}
We note that the $P_R$-VMS-POD-ROM \eqref{eq:vms-pod} is different from the VMS-POD-ROM introduced in \cite{wang2011closure}. 
Indeed, the former uses the operator $\prp$ and a constant EV coefficient, whereas the later does not use the operator $\prp$ and uses a variable EV coefficient. 
\end{remark}

We consider the full discretization of \eqref{eq:vms-pod}: 
We use the backward Euler method with a time step $\Delta t$ for the time integration and the FE space $\PP^{m}$ with $m\geq 2$ and a mesh size $h$ for the spatial discretization. 
For $k= 0, \ldots, M$, denote the approximation solution of \eqref{eq:vms-pod} at $t_k = k\Delta t$ to be $\urn = \bu_{h, r}(t_k)$ and the force at $t_k$ to be $\bff^{k} = \bff(t_k)$, respectively. 
Note that we have dropped the subscript $``h"$ in $\urn$ for clarity of notation.    
The discretized $P_R$-VMS-POD-ROM then reads: 
Find $\urn\in \bX^r$ such that 
\begin{eqnarray}
\lp  \frac{\urnp-\urn}{\Delta t} , \bphi \rp 
+ \nu \lp  \D\urnp , \D\bphi \rp 
+ b^*\lp \urnp, \urnp, \bphi\rp 
+ \alpha \, \lp \prp \nabla \urnp , \prp \nabla \bphi\rp 
= \lp \bff^{k+1}, \bphi\rp,  \nonumber  \\ 
\forall \, \bphi \in \bX^r,\, k = 0, \ldots, M-1 \hspace{.8cm}
\label{eq:vms-pod-be}
\end{eqnarray}
and the initial condition given in \eqref{eq:ic}: ${\bf u}_r^0 = \sum \limits_{j=1}^{r} ({\bf u}^0 , \bphi_j) \, \bphi_j$. 

In the sequel, we denote by $\bu^k$ and  $\bu^k_h$ the velocity solution of  \eqref{eq:nse_weak} and the FE velocity approximation of \eqref{eq:nse_fe} at $t= t_k$, respectively. 


%% file: Estimate.tex
\section{Error Estimates}\label{err}
In this section, we present the error analysis for the $P_R$-VMS-POD-ROM discretization \eqref{eq:vms-pod-be}. 
We take the FE solutions $\bu_h(\cdot, t_i)$, $i=1, \ldots, M$ as snapshots and choose $\mathcal{H} = L^2$ in the POD generation. 
The error source includes three main components: the spatial FE discretization error, the temporal discretization error, and the POD truncation error. 
We derive the error estimate in two steps: First, we gather some necessary assumptions and preliminary results in Section \ref{sec:step1}. Then, we present the main result in Section \ref{sec:step2}. 

In the sequel, \linelabel{line:C}
we assume $C$ to be a generic constant, which varies in different places, but is always independent of the finite element mesh size $h$, the finite element order $m$, the eigenvalues $\lambda_j$ and the time step size $\Delta t$. 
 
\subsection{Preliminaries}\label{sec:step1} 
We will need the following results for developing a rigorous error estimate:
\begin{assumption}[finite element error]
\label{lem:fem}
\linelabel{line:lem:fem}
We assume that the FE approximation $\bu_h$ of \eqref{eq:nse_fe} satisfies the following error estimate:
\begin{equation}
\|\bu - \bu_h\| + h \| \nabla (\bu - \bu_h) \| \leq C (h^{m+1} + \Delta t). 
\label{eq:fem_err_u}
\end{equation}
We also assume the following standard approximation property (see, e.g., page 166 in \cite{layton2008introduction}): 
\begin{equation}
\|p- q_h\| \leq C h^m.
\label{eq:fem_err_p}
\end{equation}
\end{assumption}
\begin{remark}
\label{rem:fem_err}
In chapter V of \cite{GR79}, a linearized version of the implicit (backward) Euler scheme of the NSE \eqref{eq:nse} was considered (see equation (2.2)). 
Theorem 2.2 in the same chapter proves (optimal) first order error estimates with respect to the time variable in the $L^2$ norm. 
On page 170 it is mentioned that the discretization with respect to the space variable is not considered, since it has already been throughly studied in chapter IV. 

In \cite{geveci1989convergence}, the same linearized version of the implicit (backward) Euler scheme as that in equation (2.2) in chapter V of \cite{GR79} is considered. The theorem on page 44 in \cite{geveci1989convergence} proves (optimal) first order error estimates with respect to the time variable in the $H^1$ norm. As in \cite{GR79}, the discretization with respect to the space variable was not considered. 

We also note that the implicit (backward) Euler scheme was also considered in \cite{HRII}. Section ``Time discretization" on page 765 in \cite{HRII} outlines the derivation of an optimal error estimate with respect to both space and time. For the explicit (forward) Euler scheme, an (optimal) first order error estimate with respect to the time variable was proven in \cite{Rannacher}. Higher order schemes for the time discretization of the NSE were analyzed in \cite{Baker1976galerkin,baker1982higher,GR79,HRIV}. 

Thus, although, we are not aware of any reference where estimates \eqref{eq:fem_err_u}--\eqref{eq:fem_err_p} are actually proven, the discussion above already shows that this is possible by assembling the results in \cite{geveci1989convergence,GR79,HRII}. Of course, estimates \eqref{eq:fem_err_u}--\eqref{eq:fem_err_p} have been confirmed by numerous simulations over the years. 
\end{remark}
 For the POD approximation, the following POD inverse estimate was proven in Lemma 2 in \cite{KV01}:
\begin{lemma}
Let $\bphi_i$, $i = 1, \ldots, r$, be POD basis functions, $M_{r}$ be the POD mass matrix with entries $[M_r]_{jk} = (\bphi_k, \bphi_j)$, and   
$S_{r}$ be the POD stiffness matrix with entries $[S_r]_{jk} = [M_r]_{jk} + (\nabla \bphi_k, \nabla \bphi_j)$, where $j, k = 1, \ldots, r$. 
Let $\| \cdot \|_2$ denote the matrix 2-norm.
Then, for all $\bv \in \bX^r$, the following estimates hold:
\begin{eqnarray}
\| \bv \|_{L^2} 
&\leq& \sqrt{\| M_r \|_2 \, \| S_r^{-1} \|_2 } \, \| \bv \|_{H^1} \, ,
\label{lemma_inverse_pod_1} \\
\| \bv \|_{H^1} 
&\leq& \sqrt{\| S_r \|_2 \, \| M_r^{-1} \|_2 } \, \| \bv \|_{L^2} \, .
\label{lemma_inverse_pod_2}
\end{eqnarray}
\label{lemma_inverse_pod}
\end{lemma}
Note that, since we chose $\mathcal{H}= L^2$ in the POD method, $\| M_r \|_2= \| M_r ^{-1}\|_2= 1$ in inequalities \eqref{lemma_inverse_pod_1}--\eqref{lemma_inverse_pod_2}. 

The $L^2$ norm of the POD projection error is given by \eqref{pod_error_formula} with $\mathcal{H}= L^2$. 
The $H^1$ norm of the POD projection error is given in the following lemma:
\begin{lemma}
The POD projection error in the $H^1$ norm satisfies
\begin{equation}
\frac{1}{M+1}\sum_{\ell= 0}^M
  \Big\| \bu_h(\cdot, t_{\ell}) - 
  \sum_{j=1}^r \left( \bu_h(\cdot, t_{\ell}), \bphi_j(\cdot) \right) \, \bphi_j(\cdot) 
  \Big\|_{1}^2
  = \sum\limits_{j= r+1}^d \|\bphi_j\|_{1}^2\, \lambda_j .
 \label{eq:poderrH1} 
\end{equation}
\label{pr:poderrH1}
\end{lemma}
Note that the POD projection error for continuous functions, i.e., the error in the $L^2(0, T; H^1(\Om))$ norm, has been proven in \cite{singler2012new} (Theorem 2, page 17). We consider the POD of a discrete function and derive the time averaged POD projection error in the $H^1$ norm as follows: 
\begin{proof}
Let $Y= \left[\bu_h(\cdot, t_{0}), \bu_h(\cdot, t_{1}), \ldots, \bu_h(\cdot, t_{M})\right]$ be the snapshot matrix. 
A necessary optimality condition of the POD basis is given by the following eigenvalue problem (see, e.g., \cite{kunisch1999}):
\begin{equation}
\frac{1}{M+1}YY^\intercal \bphi_j = \lambda_j \bphi_j. 
\label{eq:pod_eig}
\end{equation}
The POD projection error in the $H^1$ norm satisfies 
{\allowdisplaybreaks
\begin{eqnarray}
&&\frac{1}{M+1}\sum_{\ell= 0}^M
  \Big\| \bu_h(\cdot, t_{\ell}) - 
  \sum_{j=1}^r \left( \bu_h(\cdot, t_{\ell}), \bphi_j(\cdot) \right) \, \bphi_j(\cdot) 
  \Big\|_{1}^2 
  \nonumber \\
&=&  \frac{1}{M+1}\sum_{\ell= 0}^M
 \Big\|\sum_{j=r+1}^d (\bu_h(\cdot, t_{\ell}), \bphi_j) \bphi_j \Big\|_1^2  
 \nonumber \\
 &=&  \frac{1}{M+1}\sum_{\ell= 0}^M 
 \lp \sum_{j=r+1}^d  (\bu_h(\cdot, t_{\ell}), \bphi_j) \bphi_j, \sum_{k=r+1}^d  (\bu_h(\cdot, t_{\ell}), \bphi_k) \bphi_k \rp_1 
 \nonumber \\
  &=&  \frac{1}{M+1}\sum_{\ell= 0}^M  \sum_{j=r+1}^d  \sum_{k=r+1}^d
 (\bu_h(\cdot, t_{\ell}), \bphi_j) (\bu_h(\cdot, t_{\ell}), \bphi_k) (\bphi_j,   \bphi_k)_1
 \nonumber \\
     &=&  \sum_{j=r+1}^d  \sum_{k=r+1}^d
 \lp \frac{1}{M+1}\sum_{\ell= 0}^M \Big( \bu_h(\cdot, t_{\ell}), \bphi_j \Big) \bu_h(\cdot, t_{\ell}), \bphi_k \rp 
 (\bphi_j,   \bphi_k)_1
 \nonumber \\
    &=&  \sum_{j=r+1}^d  \sum_{k=r+1}^d
 \lp \frac{1}{M+1}YY^\intercal \bphi_j, \bphi_k \rp 
 (\bphi_j,   \bphi_k)_1
 \nonumber \\
  &\stackrel{\eqref{eq:pod_eig}}{=}&  \sum_{j=r+1}^d  \sum_{k=r+1}^d
 \lp \lambda_j \bphi_j, \bphi_k \rp 
 (\bphi_j,   \bphi_k)_1
 \nonumber \\
   &=&  \sum_{j=r+1}^d  \sum_{k=r+1}^d
 \lambda_j \delta_{jk} 
 (\bphi_j,   \bphi_k)_1
 = \sum_{j=r+1}^d \lambda_j \|\bphi_j\|_1^2,
\end{eqnarray} 
}
which proves \eqref{eq:poderrH1}. \qed 
\end{proof}
We define the $L^2$ projection of $\bu$, $P_r \bu$, from ${\bf L}^2$ to $\bX^r$ as follows: 
\begin{equation}
\left( \bu- P_r\bu, \bphi_r\right) = 0,\qquad  \forall\, \bphi_r\in \bX^r. 
\label{eq:L2proj}
\end{equation}
We have the following error estimate of the $L^2$ projection:
\begin{lemma}
For any $\bu^k \in \bX$, its $L^2$ projection, $\bw_r^k = P_r \bu^k$, satisfies the following error estimates:
\begin{equation}
	   \frac{1}{M+1} \sum_{k=0}^{M} \left\| \bu^k- \bw_r^k  \right\|^2 
	\leq C \lp h^{2m+2} + \Delta t^2 + \sum _{j = r+1}^{d} \lambda_j \rp ,
\label{eq:error_eta}
\end{equation}
\begin{equation}
	   \frac{1}{M+1} \sum_{k=0}^{M} \left\| \nabla \lp \bu^k- \bw_r^k \rp \right\|^2 
	\leq C \lp h^{2m} + \|S_r\|_2 h^{2m+2}+ (1+\|S_r\|_2 )\Delta t^2 + \sum _{j = r+1}^{d}\|\bphi_j\|_1^2\, \lambda_j  \rp. 
\label{eq:error_geta}
\end{equation}
\label{lemma_proj}
\end{lemma} 
\begin{proof}
By the definition of the $L^2$ projection \eqref{eq:L2proj}, we have 
\begin{equation}
\left\| \bu^k - \bw_r^k \right\|^2
= \left( \bu^k - \bw_r^k, \bu^k - \bw_r^k \right) 
\stackrel{ \eqref{eq:L2proj} }{=} \left( \bu^k - \bw_r^k, \bu- \bv_r^k \right), \qquad \forall \, \bv_r^k \in \bX^r. 
\label{lemma_L2_1}
\end{equation}
Using the Cauchy-Schwarz inequality in \eqref{lemma_L2_1}, we get
\begin{equation}
\left\| \bu^k-\bw_r^k \right\| \leq \|\bu^k-\bv_r^k\|, \qquad \forall \bv_r^k \in \bX^r.
\label{lemma_L2_2}
\end{equation}
Decompose $\bu^k - \bv_r^k = (\bu^k - \bu^k_h) + (\bu^k_h-\bv_r^k)$, where $\bu^k_h$ is the corresponding FE approximation. 
Choosing $\bv_r^k = P_r \bu^k_h :=  \sum\limits_{j=1}^r \left( \bu^k_h, \bphi_j \right)\bphi_j $ in \eqref{lemma_L2_2}, by the triangle inequality, Assumption \ref{lem:fem}, and the POD projection error estimate \eqref{pod_error_formula}, we have
\begin{eqnarray}
	   \frac{1}{M+1} \sum_{k=0}^{M} \left\| \bu^k- \bw_r^k \right\|^2 
&\leq& \frac{1}{M+1} \sum_{k=0}^{M} \lp \left\| \bu^k- \bu^k_h \right\| + \left \| \bu^k_h - P_r \bu^k_h \right\| \rp^2 \nonumber \\
&\leq& C \lp h^{2m+2} + \Delta t^2 + \sum _{j = r+1}^{d} \lambda_j \rp,  
\end{eqnarray}
which proves error estimate \eqref{eq:error_eta}. 

Using the triangle inequality,  Assumption \ref{lem:fem}, the POD inverse estimate \eqref{lemma_inverse_pod_2} and Lemma \ref{pr:poderrH1}, we obtain    
\begin{eqnarray}
&& \frac{1}{M+1} \sum_{k=0}^{M} \left\| \nabla \lp \bu^k- \bw_r^k \rp \right\|^2 \nonumber \\
&\leq& \frac{1}{M+1} \sum_{k=0}^{M} \lp \left\| \nabla \lp \bu^k- \bu^k_h \rp \right\| + 
\left \| \nabla \lp \bu^k_h - P_r \bu^k_h \rp \right\| +
\left\|  \nabla \lp P_r \bu^k_h - \bw_r^k \rp \right\| \rp^2 \nonumber \\
&\leq&C \lp h^{2m} +\Delta t^2 + \sum_{j=r+1}^d\|\varphi_j\|_1^2 \, \lambda_j 
+\|S_r\|_2\frac{1}{M+1}\sum_{k=0}^M \left\|P_r \bu^k_h - \bw_r^k\right\|^2 \rp
\nonumber \\
&{\leq}&C \lp h^{2m} +\Delta t^2 + \sum_{j=r+1}^d\|\varphi_j\|_1^2 \, \lambda_j 
+\|S_r\|_2 \frac{1}{M+1}\sum_{k=0}^M\left\|\bu^k_h - \bu^k\right\|^2 \rp 
\qquad (\bw_r^k = P_r\bu^k)
\nonumber \\
&\leq& C \lp h^{2m} + \|S_r\|_2 h^{2m+2}+ (1+\|S_r\|_2 )\Delta t^2 + \sum _{j = r+1}^{d} \|\bphi_j\|_1^2\, \lambda_j  \rp, 
\end{eqnarray}
which proves error estimate \eqref{eq:error_geta}. \qed

\end{proof}
We assume that the following estimates, which are similar to \eqref{eq:error_eta} and \eqref{eq:error_geta}, are also valid:
\begin{assumption} 
For any $\bu^k \in \bX$, its $L^2$ projection, $\bw_r^k = P_r \bu^k$, satisfies the following error estimates:
\begin{eqnarray}
&& \lno  \bu^k - \bw_r^k \rno  \leq
C \, \left(
h^{m+1}
+ \Delta t
+\sqrt{\sum\limits_{j=r+1}^d \lambda_j}\,
\right) \, , 
\label{eq:error_eta_no_sum} \\
&& \lno  \nabla \lp \bu^k - \bw_r^k \rp \rno  \leq
C \, \left(
h^{m}   
+ \sqrt{\|S_r\|_2}h^{m+1}
+ \sqrt{1+\|S_r\|_2}\Delta t
+ \sqrt{\sum _{j = r+1}^{d} \|\bphi_j\|_1^2\, \lambda_j}\,
\right) \, . 
\label{eq:error_geta_no_sum}
\end{eqnarray}
\label{assumption_poderror}
\end{assumption}
\begin{remark}
The assumption that \eqref{eq:error_eta_no_sum} and \eqref{eq:error_geta_no_sum} hold is quite natural.
It simply says that, in the POD truncation error \eqref{pod_error_formula} and \eqref{eq:poderrH1}, no individual term is much larger than the other terms in the sums.

We also mention that estimates \eqref{eq:error_eta_no_sum} and \eqref{eq:error_geta_no_sum} would follow directly from the POD truncation error estimates \eqref{pod_error_formula} and \eqref{eq:poderrH1} if we discarded the $\frac{1}{M + 1}$ factor in those estimates.
This could be accomplished simply by dropping the $\frac{1}{M + 1}$ factor from the snapshot correlation matrix $K$. 
In fact, this approach is used in, e.g., \cite{KV02,volkwein2011model}.
We note, however, that by dropping the $\frac{1}{M + 1}$ from the correlation matrix $K$ would most likely increase the magnitudes of the eigenvalues on the RHS of the POD truncation error estimates \eqref{pod_error_formula} and \eqref{eq:poderrH1}. 
\end{remark}

\begin{lemma}[see Lemma 13 and Lemma 14 in \cite{layton2008introduction}]
\label{lem:b*}
For any functions $\bu, \, \bv, \bw \in \bX$, the skew-symmetric trilinear form $b^*(\cdot, \cdot, \cdot)$ satisfies
\begin{equation}
b^*(\bu, \bv, \bv) = 0,
\label{eq:b_skewsymm}
\end{equation}
\begin{equation} 
b^*(\bu, \bv , \bw)\leq C \|\nabla \bu\| \|\nabla \bv\| \|\nabla \bw\|, 
\label{eq:b_bound_0}
\end{equation}
and a sharper bound 
\begin{equation} 
b^*(\bu, \bv , \bw)\leq C \sqrt{\|\bu\| \|\nabla \bu\|}\|\nabla \bv\| \|\nabla \bw\|. 
\label{eq:b_bound_1}
\end{equation}
\end{lemma}

We have the following stability result for the $P_R$-VMS-POD-ROM \eqref{eq:vms-pod-be}: 
\begin{lemma}
The solution of \eqref{eq:vms-pod-be} satisfies the following bound:
\begin{equation}
\lno  \bu_r^{M} \rno^2  
+ \nu \Delta t \sNn \lno \nabla \urnp \rno^2
\leq \lno  \bu_r^0 \rno^2 
+ \frac{\Delta t}{\nu} \sNn \lno  {\bff}^{k+1} \rno_{-1}^2 .
\label{stability:c}
\end{equation}
\label{th:stability}
\end{lemma}
\begin{proof}
Choosing $\bphi:=\urnp$ in \eqref{eq:vms-pod-be} and noting that $b^*(\urnp, \urnp , \urnp) = 0$ by \eqref{eq:b_skewsymm}, we obtain
\begin{equation}
\lp {\urnp-\urn} , \urnp \rp
+ {\nu \Delta t} \lp \D\urnp , \D\urnp \rp
+ \alpha{\Delta t} \, \lp \prp \nabla \urnp , \prp \nabla \urnp\rp 
= {\Delta t}\lp {\bff}^{k+1}, \urnp\rp.  
\label{eq:egy1}
\end{equation}
Using the Cauchy-Schwarz inequality, Young's inequality and the fact that the last term on the LHS of \eqref{eq:egy1} is positive yields
\begin{eqnarray}
\frac{1}{2} \lno  \urnp \rno ^2 - \frac{1}{2} \lno \urn \rno^2 
+ \nu \Delta t \lno \nabla \urnp \rno^2
\leq {\Delta t}\lp {\bff}^{k+1}, \urnp\rp.  
\label{eq:th:stability_1}
\end{eqnarray}
Applying the Cauchy-Schwarz inequality and Young's inequality on the RHS of \eqref{eq:th:stability_1}, 
we get 
\begin{eqnarray}
\frac{1}{2} \lno  \urnp \rno ^2 - \frac{1}{2} \lno \urn \rno^2 
+ \nu \Delta t \lno \nabla \urnp \rno^2
\leq \frac{\Delta t}{2 \nu} \lno {\bff}^{k+1} \rno_{-1}^2 + \frac{\nu \Delta t}{2} \lno \nabla \urnp \rno^2.
\label{stability:1} 
\end{eqnarray}
Then, the stability estimate \eqref{stability:c} follows by summing \eqref{stability:1} from 0 to $M-1$. \qed
\end{proof}

\begin{lemma}
\label{lem:ubounds}
The a priori stability estimate in Lemma~\ref{th:stability} yields the following bounds: 
\begin{equation}
\| {\bf u}_{r}^{k+1} \|^2
\leq \nu^{-1} \, \vertiii {\bf f}_{2,-1}^{2} + \|\bu_r^0\|^2,\qquad \forall\, k=0, \ldots, M-1.
\label{eqn:3.42_rhs_5a}
\end{equation}
\end{lemma}
\subsection{Main Results}\label{sec:step2}
We are ready to derive the main result of this section, which provides the error estimates for the $P_R$-VMS-POD-ROM \eqref{eq:vms-pod-be}.
\begin{theorem}
Under the regularity assumption of the exact solution (Assumption \ref{assumption_regularity}), the assumption on the FE approximation (Assumption \ref{lem:fem}) and the assumption on the POD projection error (Assumption \ref{assumption_poderror}), the solution of the $P_R$-VMS-POD-ROM \eqref{eq:vms-pod-be} satisfies the following error estimate: There exists $\Delta t^*>0$ such that the inequality
\begin{eqnarray}
\label{eq:theorem_error_1}
&& \lno {\bf u}^M - {\bf u}_r^M \rno^2
+ \nu
\Delta t \, \sum \limits_{k = 0}^{M-1} \lno \nabla \lp {\bf u}^{k+1} - {\bf u}_{r}^{k+1}\rp \rno^2
\nonumber \\
&\leq& C
\Bigg( \lp 1+\|S_r\|_2+\|S_R\|_2 \rp \Delta t^2 
+ h^{2m}
+ \lp 1+\|S_r\|_2+\|S_R\|_2 \rp h^{2m+2}
\\ \nonumber
&&+\sum_{j= r+1}^d \lp 1+\|\bphi_j\|_1^2 \rp \, \lambda_j 
 + \sum_{j=  R+1}^d \|\bphi_j\|_1^2 \, \lambda_j
\Bigg)
\end{eqnarray}
holds for all $\Delta t< \Delta t^*$.
\label{theorem_error}
\end{theorem}
\begin{proof}
We start deriving the error bound by splitting the error into two terms: 
\begin{equation}
\unp-\urnp = \lp\unp - \wrnp \rp - \lp \urnp - \wrnp \rp = \etanp - \prnp. 
\label{eq:error}
\end{equation}
The first term, $\etanp=\unp - \wrnp$, represents the difference between $\unp$ and its $L^2$ projection on $\bX^r$, which has been bounded in Lemma~\ref{lemma_proj}. 
The second term, $\prnp$, is the remainder. 

Next, we construct the error equation. We first evaluate the weak formulation of the NSE \eqref{eq:nse_weak} at $t=\tnp$ and 
let $\bv = \bphi_r$, then subtract the $P_R$-VMS-POD-ROM \eqref{eq:vms-pod-be} from it. 
We obtain
\begin{eqnarray}
&&\lp \unp_t , \bphi_r \rp - \lp \frac{\urnp-\urn}{\Delta t} , \bphi_r \rp
+ \nu \lp \D\unp - \D\urnp, \D\bphi_r \rp
+ b^*\lp \unp,\unp,\bphi_r \rp  \nonumber \\ 
&& - b^*\lp \urnp, \urnp, \bphi_r \rp
- \lp p , \nabla \cdot \bphi_r \rp
- \alpha \, \lp \prp \nabla \urnp , \prp \nabla \bphi_r \rp
= 0, \quad
\forall \, \bphi_r \in \bX^r.
\label{eq:error-eqn}
\end{eqnarray}
By subtracting and adding the difference quotient term, $\lp\frac{\unp-\un}{\Delta t}, \bphi_r \rp$, in \eqref{eq:error-eqn}, and applying the decomposition \eqref{eq:error}, we have, for any $\bphi_r \in \bX^r$, 
\begin{eqnarray}
&&\lp \unp_t - \frac{\unp-\un}{\Delta t}, \bphi_r \rp 
+\frac{1}{\Delta t}\lp \etanp - \prnp, \bphi_r \rp
-\frac{1}{\Delta t} \lp \etan - \prn, \bphi_r \rp
+\nu \lp \D\lp \etanp-\prnp\rp , \D\bphi_r \rp  \nonumber \\ 
&&+ b^*\lp \unp,\unp,\bphi_r \rp
- b^* \lp \urnp, \urnp, \bphi_r \rp
 - \lp p , \nabla \cdot \bphi_r \rp
- \alpha \, \lp \prp \nabla \urnp , \prp \nabla \bphi_r \rp
= 0. 
\label{eq:error-eqn2}
\end{eqnarray}
Note that \eqref{eq:L2proj} implies that $\lp \etan, \bphi_r \rp = 0$ and $\lp \etanp, \bphi_r \rp = 0$. 
Choosing $\bphi_r = \prnp$ in \eqref{eq:error-eqn2} and letting $\rn = \utnp - \frac{\unp-\un}{\Delta t}$, 
we obtain
\begin{eqnarray}
&&\frac{1}{\Delta t}\lp \prnp, \prnp\rp  -\frac{1}{\Delta t}\lp \prn, \prnp\rp 
+\nu  \lp \, \D \prnp , \D \prnp  \rp \nonumber \\
&&= \lp \rn, \prnp \rp 
+ \nu \lp \D \etanp , \D \prnp  \rp  
+ b^*\lp \unp,\unp, \prnp\rp  
\nonumber \\ 
&& - b^*\lp \urnp, \urnp, \prnp\rp 
- \lp p , \nabla \cdot \prnp \rp
- \alpha \, \lp \prp \nabla \urnp , \prp \nabla \prnp\rp   \, .
\label{eq:error-eqn3}
\end{eqnarray}
First, we estimate the LHS of \eqref{eq:error-eqn3} by applying the Cauchy-Schwarz inequality and Young's inequality:
\begin{eqnarray}
\text{LHS} &=& \frac{1}{\Delta t}\lno \prnp\rno ^2 -\frac{1}{\Delta t}\lp \prn, \prnp\rp 
+\nu\lno \D \prnp \rno ^2 \nonumber \\
&\geq& \frac{1}{2 \Delta t}\lp\lno \prnp\rno ^2 - \lno \prn\rno ^2\rp+\nu\lno \D \prnp \rno ^2 \, .
\label{eq:error-rhs}
\end{eqnarray} 
Multiplying by $2\Delta t$ both sides of inequality \eqref{eq:error-rhs} and using the result in \eqref{eq:error-eqn3}, we obtain
\begin{eqnarray}
&&\lno \prnp\rno ^2 - \lno \prn\rno ^2
+2\nu\Delta t\lno \D \prnp \rno ^2 \nonumber \\
&\leq& 2\Delta t\lp \rn, \prnp \rp 
+ 2\nu\Delta t \lp \D \etanp , \D \prnp  \rp  
+ 2 \Delta t\, b^*\lp \unp,\unp, \prnp\rp  \nonumber \\ 
&& 
- 2\Delta t\, b^*\lp \urnp, \urnp, \prnp\rp 
- 2\Delta t \, \lp p , \nabla \cdot \prnp \rp
- 2 \alpha\Delta t  \, \lp \prp \nabla \urnp , \prp \nabla \prnp\rp  \, .
\label{eq:ineqn1}
\end{eqnarray} 
Next, we estimate the terms on the RHS of \eqref{eq:ineqn1} one by one.
Using the Cauchy-Schwarz inequality and Young's inequality, we get
\begin{equation}
\lp \rn, \prnp \rp \leq
\lno \rn \rno_{-1} \, \lno \nabla \prnp\rno  \leq
\frac{c_1^{-1}}{4} \lno \rn \rno_{-1}^2 +  c_1 \lno \nabla \prnp\rno ^2\, ,
\label{eq:error-lhs-1} 
\end{equation}
\begin{equation}
\nu \lp \D \etanp , \D \prnp \rp \leq
\nu \lno \D \etanp \rno \, \lno \D \prnp \rno  \leq
 \frac{c_2^{-1}\nu}{4} \lno \D \etanp \rno ^2 
+ c_2\nu \lno \D \prnp \rno ^2\, .
\label{eq:error-lhs-4} 
\end{equation}
The nonlinear terms in \eqref{eq:ineqn1} can be written as follows:
\begin{eqnarray}
& &b^*\lp \unp,\unp, \prnp\rp  - b^*\lp \urnp, \urnp, \prnp\rp  \nonumber \\
&=&b^*\lp \urnp,\etanp-\prnp,\prnp\rp  + b^*\lp \etanp-\prnp,\unp,\prnp\rp \nonumber \\
&=&b^*\lp \urnp,\etanp,\prnp\rp  + b^*\lp \etanp,\unp,\prnp\rp  - b^*\lp \prnp,\unp,\prnp\rp ,
\label{eq:error-lhs-56} 
\end{eqnarray}
where we have used $b^*\lp \unp,\prnp, \prnp\rp = 0$, which follows from \eqref{eq:b_skewsymm}. 
Next, we estimate each term on the RHS of \eqref{eq:error-lhs-56}.
Since ${\bf u}_{r}^{k+1}, \etanp, \prnp \in \bX$, we can apply the standard bounds for the trilinear form $b^*\lp \cdot, \cdot, \cdot \rp$ and use Young's inequality:
\begin{eqnarray}
b^*\lp \urnp,\etanp,\prnp\rp  &\stackrel{\eqref{eq:b_bound_1}}{\leq}&
C \, \lno \urnp\rno^{1/2}  \, \lno \nabla \urnp\rno^{1/2}  \,  
\lno \nabla \etanp \rno  \, \lno \nabla \prnp\rno  \hfill \nonumber \\
&\leq&
\frac{1}{4 c_3}C^2 \, \lno \urnp \rno \,  \lno \nabla \urnp \rno \,  \lno \nabla \etanp\rno ^2 
+  c_3 \lno \nabla \prnp\rno ^2 \, ;
\label{eq:error-lhs-56-1} 
\end{eqnarray}
\begin{eqnarray}
b^*\lp \etanp,\unp,\prnp\rp  &\stackrel{\eqref{eq:b_bound_0}}{\leq}&
C \lno \nabla \etanp\rno \, \lno \nabla \unp\rno  \, \lno \nabla \prnp\rno  \nonumber \\
&\leq&
\frac{1}{4 c_4} C^2  \lno \nabla \unp\rno ^2  \, \lno \nabla \etanp\rno ^2  + c_4 \lno \nabla \prnp\rno ^2 \, ;
\label{eq:error-lhs-56-2} 
\end{eqnarray}
\begin{eqnarray}
b^*\lp \prnp,\unp,\prnp\rp &\stackrel{\eqref{eq:b_bound_1}}{\leq}&
C \lno  \prnp\rno ^{\half}\, \lno \nabla \prnp\rno ^{\half} \, \lno \nabla \unp\rno  \, \lno \nabla \prnp\rno  \nonumber \\
&=& C \lno  \prnp\rno ^{\half} \, \lno \nabla \unp\rno  \, \lno \nabla \prnp\rno ^{\frac{3}{2}} \nonumber \\
&\leq&
C \frac{c_5^{-3}}{4}  \lno \nabla \unp\rno ^4 \, \lno \prnp\rno ^2  + 
C \frac{3 c_5}{4} \lno \nabla \prnp\rno ^2 \, .
\label{eq:error-lhs-56-3} 
\end{eqnarray}
Since $\prnp \in \bX^r \subset \bV^h$, the pressure term on the RHS of \eqref{eq:ineqn1} can be written as 
\begin{equation}
- \lp p , \nabla \cdot \prnp \rp
= - \lp p - q_h, \nabla \cdot \prnp \rp \, ,
\label{eq:error-lhs-pressure_1} 
\end{equation}
where $q_h$ is any function in $Q^h$.
Thus, the pressure term can be estimated as follows by the Cauchy-Schwarz inequality and Young's inequality:
\begin{equation}
- \lp p , \nabla \cdot \prnp \rp
\leq \frac{1}{4 c_6} \, \lno p - q_h \rno^2
+ c_6 \, \lno \nabla \prnp \rno^2 .
\label{eq:error-lhs-pressure_2} 
\end{equation}
The last term on the RHS of \eqref{eq:ineqn1} can be estimated as follows:
{\allowdisplaybreaks
\begin{eqnarray}
&&-\alpha \, \lp \prp \nabla \urnp , \prp \nabla \prnp\rp \nonumber \\
&=& \alpha \lp \prp \nabla \unp - \prp \nabla \urnp, \prp \nabla \prnp \rp 
- \alpha \lp \prp \nabla \unp, \prp \nabla \prnp \rp
\nonumber \\
&=& \alpha \lp \prp \nabla \etanp, \prp \nabla \prnp \rp 
- \alpha \lp \prp \nabla \prnp, \prp \nabla \prnp \rp 
- \alpha \lp \prp \nabla \unp, \prp \nabla \prnp \rp
\nonumber \\
&\leq&
\alpha \lno \prp \nabla \etanp \rno  \cdot \lno  \prp \nabla \prnp \rno
- \alpha \lno \prp \nabla \prnp\rno ^2
- \alpha \lp \prp \nabla \unp, \prp \nabla \prnp \rp  \nonumber \\
&\leq&
\alpha \lp \lno \prp \nabla \etanp \rno^2 + \frac{1}{4} \lno  \prp \nabla \prnp \rno^2 \rp
- \alpha \lno \prp \nabla \prnp\rno ^2
+ \alpha \lp \lno \prp \nabla \unp \rno^2 + \frac{1}{4} \lno \prp \nabla \prnp \rno^2 \rp  \nonumber \\
&\leq& 
\alpha \lno \prp \nabla \etanp \rno ^2
- \frac{\alpha}{2} \lno \prp \nabla \prnp\rno ^2 
+ \alpha \lno \prp \nabla \unp \rno^2 \, .
\label{eq:error-lhs-7} 
\end{eqnarray}
}
Note that, since $P_R$ is the $L^2$ projection of ${\bL}^2$ on $\bL^R$, we get
\begin{eqnarray*}
\lno \prp \nabla \etanp \rno  \leq \lno \nabla \etanp \rno  \, .
\end{eqnarray*}
Choosing $c_1=c_3=c_4=c_6=\frac{\nu}{12}$, $c_2 = \frac{1}{12}$ and $c_5=\frac{\nu}{9C}$, 
then substituting the above inequalities in \eqref{eq:ineqn1}, we obtain
\begin{eqnarray}
&&\lno \prnp\rno ^2 - \lno \prn\rno ^2 + 
 \nu \Delta t \lno \D \prnp \rno ^2 + \alpha  \Delta t \lno \prp \nabla \prnp\rno^2
\nonumber \\
&\leq& 
\frac{6\Delta t}{\nu} \lno \rn \rno_{-1}^2 
+ 6\nu \Delta t \lno \nabla \etanp \rno ^2          
+ \frac{6 \Delta t }{\nu} C^2 \lno \urnp \rno \, \lno \D \urnp \rno \,  \lno \nabla \etanp \rno^2
+ \frac{6 \Delta t}{\nu} C^2 \lno \nabla \unp \rno^2  \, \lno \nabla \etanp \rno^2  
\nonumber \\
&+& \frac{C^4 9^3 \nu^{-3} \Delta t }{2} \lno \nabla \unp\rno ^4 \, \lno \prnp\rno ^2
+ \frac{6 \Delta t}{\nu} \, \lno p - q_h \rno^2
+ 2 \alpha \Delta t \lno \nabla \etanp \rno^2
+ 2 \alpha \Delta t \lno \prp \nabla \unp \rno^2 \, .
\label{eq:error-eqn4}
\end{eqnarray}
Summing \eqref{eq:error-eqn4} from $k=0$ to $k=M-1$, we have 
\begin{eqnarray}
&&\lno \prm \rno^2 + \nu \Delta t \sNn \lno \D \prnp \rno ^2 + \alpha  \Delta t \sNn \lno \prp \nabla \prnp\rno^2
\nonumber \\
&\leq&  \lno \prz \rno^2 + \frac{6\Delta t}{\nu} \sNn \lno \rn \rno_{-1}^2 
+ 6\nu \Delta t \sNn \lno \nabla \etanp \rno ^2        
+ \frac{6 \Delta t }{\nu} C^2 \sNn  \lno \urnp \rno \, \lno \D \urnp \rno \,  \lno \nabla \etanp \rno^2 \nonumber \\
&+&\frac{6 \Delta t}{\nu} C^2 \sNn \lno \nabla \unp \rno^2  \, \lno \nabla \etanp \rno^2 
+\frac{C^4 9^3 \nu^{-3} \Delta t }{2} \sNn \lno \nabla \unp\rno ^4 \, \lno \prnp\rno ^2 \nonumber \\
&+& \frac{6 \Delta t}{\nu} \, \sNn \lno p - q_h \rno^2 
+ 2 \alpha \Delta t \sNn \lno \nabla \etanp \rno^2
+ 2 \alpha \Delta t \sNn \lno \prp \nabla \unp \rno^2 \, . 
\label{eqn:error-eqn5}
\end{eqnarray}
Next, we estimate each term on the RHS of \eqref{eqn:error-eqn5}.

The first term vanishes since $\bu_r^0 = \bw_r^0$ (see \eqref{eq:ic}). 

By using the Poincar\'{e}-Friedrichs' inequality, the second term on the RHS of \eqref{eqn:error-eqn5} can be estimated as follows (see, e.g., \cite{wang2011variational}): 
\begin{equation}
\Delta t\, \sNn \lno \rn \rno_{-1}^2 
\leq C\Delta t\, \sNn \lno \rn \rno^2 
\leq C\Delta t^2 \, \lno \bu_{tt} \rno^{2}_{2,2} \, .
\label{eq:proof1}
\end{equation}
Using \eqref{eq:error_geta}, the third and eighth terms on the RHS of \eqref{eqn:error-eqn5} can be estimated as follows:
\begin{eqnarray}
\Delta t \sum \limits_{k=0}^{M-1} \| \nabla \eta^{k+1}\|^2
\leq C \, \left(
h^{2m}   
+ \|S_r\|_2 h^{2m+2}
+ \lp 1+\|S_r\|_2 \rp \Delta t^2
+ \sum\limits_{j=r+1}^d \|\bphi_j\|_1^2\, \lambda_j\,
\right).
\label{eqn:3.42_rhs_4_and_8}
\end{eqnarray}
To estimate the fourth term on the RHS of \eqref{eqn:error-eqn5}, we use Lemma~\ref{lem:ubounds}
\begin{eqnarray}
&& \Delta t \sum \limits_{k=0}^{M-1} \lno {\bf u}_{r}^{k+1} \rno \, \lno \nabla {\bf u}_{r}^{k+1} \rno \, \lno \nabla \eta^{k+1} \rno^2
\nonumber \\
&\stackrel{\eqref{eqn:3.42_rhs_5a}}{\leq}& \left( \nu^{-1/2} \, \vertiii {\bf f}_{2,-1} + \|\bu_r^0\| \right) \, 
\Delta t \sum \limits_{k=0}^{M-1} \lno \nabla {\bf u}_{r}^{k+1} \rno \, \lno \nabla \eta^{k+1} \rno^2 
\nonumber \\
&\stackrel{\eqref{eq:error_geta_no_sum}}{\leq}& C \, \left( \nu^{-1/2} \, \vertiii {\bf f}_{2,-1} + \|\bu_r^0\| \right) \, 
\Delta t \sum \limits_{k=0}^{M-1} \lno \nabla {\bf u}_{r}^{k+1} \rno \, 
\Big(
h^{2m}   
+ \|S_r\|_2 h^{2m+2}
+ \lp 1+\|S_r\|_2 \rp \Delta t^2 \nonumber \\
&&\hspace{2cm}+ \sum\limits_{j=r+1}^d \|\bphi_j\|_1^2\, \lambda_j\,
\Big).
\label{eqn:3.42_rhs_5d}
\end{eqnarray}
We note that we used estimate \eqref{eq:error_geta_no_sum} in the derivation of \eqref{eqn:3.42_rhs_5d};
using \eqref{eq:error_geta} would not have been enough for the asymptotic convergence of \eqref{eqn:3.42_rhs_5d}.

The fifth term on the RHS of \eqref{eqn:error-eqn5} can be bounded as follows:
\begin{eqnarray}
&& \Delta t \sum \limits_{k=0}^{M-1} \lno \nabla {\bf u}^{k+1} \rno^2 \, \lno \nabla \eta^{k+1} \rno^{2}
\nonumber \\
&\stackrel{\eqref{eq:error_geta_no_sum}}{\leq}&
C \, \Delta t \sum \limits_{k=0}^{M-1} \lno \nabla {\bf u}^{k+1} \rno^2 \, 
\Big(
h^{2m}   
+ \|S_r\|_2 h^{2m+2}
+ \lp 1+\|S_r\|_2 \rp \Delta t^2
+ \sum\limits_{j=r+1}^d \|\bphi_j\|_1^2\, \lambda_j\,
\Big). 
\label{eqn:3.42_rhs_6}
\end{eqnarray}

The seventh term on the RHS of \eqref{eqn:error-eqn5} has been bounded by the approximation property \eqref{eq:fem_err_p} in Assumption~\ref{lem:fem}.  

Using \eqref{eq:LR}, we have the following error bound of the last term on the RHS of \eqref{eqn:error-eqn5}: 
\begin{eqnarray}
&& \Delta t \, \sNn \lno \prp \nabla \unp \rno^2  
=  \Delta t \, \sNn \lno \nabla \unp - \pr \nabla \unp \rno^2   \nonumber \\
&&  \stackrel{}{\leq}  C \, \asNn 
\inf_{\bv_R \in \bX^R} \lno \nabla \unp  - \nabla \bv_R \rno^2
\leq C \,  \asNn \lno \nabla \unp  - \nabla \wRnp \rno^2
\nonumber \\ 
&&\stackrel{\eqref{eq:error_geta}}{\leq} C \lp h^{2m}   
+ \|S_R\|_2 h^{2m+2}
+ \lp 1+\|S_R\|_2 \rp \Delta t^2
+ \sum\limits_{j=R+1}^d \|\bphi_j\|_1^2\, \lambda_j\,
 \rp. 
\label{eq:proof4}
\end{eqnarray}

Collecting \eqref{eq:proof1}-\eqref{eq:proof4} and letting
$d = C(6\nu + 2\alpha)+6 C^3 \nu^{-1} (\nu^{-\frac{1}{2}} \vertiii{f}_{2, -1} +\|\bu^0_r\|) \vertiii{\nabla \bu_r}_{1, 0} + 6C^3\nu^{-1}\vertiii{\nabla \bu_r}_{2, 0}^2$,
$d_1 =  \frac{C^4 9^3 \nu^{-3}}{2}$,
$d_2 = 6 C \nu^{-1} (\lno \bu_{tt} \rno^{2}_{2,2}+1) + 2C \alpha +d$,
$d_3 = 6 C \nu^{-1} + 2C \alpha +d$, and 
$d_4 = 2\alpha C$,
equation \eqref{eqn:error-eqn5} becomes

\begin{eqnarray}
&& \lno \prm \rno^2
+\nu \Delta t \, \sum \limits_{k = 0}^{M-1} \lno \nabla \phi_{r}^{k+1} \rno^2
+ \alpha  \Delta t \sNn \lno \prp \nabla \prnp\rno^2
\nonumber \\
&\leq& d_1 \Delta t \sNn \lno \nabla \unp\rno ^4 \, \lno \prnp\rno ^2 + d_2 \Delta t^2 
+ d \|S_r\|_2 \Delta t^2 
+ d_3 h^{2m}
+ d \|S_r\|_2 h^{2m+2}
 \nonumber \\
&+& d\sum_{j= r+1}^d \|\bphi_j\|_1^2 \, \lambda_j +
d_4 \lp \|S_R\|_2 \Delta t^2 +  \|S_R\|_2 h^{2m+2} + \sum_{j=  R+1}^d \|\bphi_j\|_1^2 \, \lambda_j \rp.
\label{eq:proof5}
\end{eqnarray}

If $\Delta t< \Delta t^*:= d_1\vertiii{\nabla \bu}^4_{4, 0}$, the discrete Gronwall lemma (see Lemma 27 in \cite{layton2008introduction} and also \cite{HRIV}) implies the following inequality: 
\begin{eqnarray}
&& \lno \prm \rno^2
+ \nu
\Delta t \, \sum \limits_{k = 0}^{M-1} \lno \nabla \phi_{r}^{k+1} \rno^2
+ \alpha  \Delta t \sNn \lno \prp \nabla \prnp\rno^2
\nonumber \\
&\leq& C^*
\Bigg( 
d_2 \Delta t^2 
+ d \|S_r\|_2 \Delta t^2 
+ d_3 h^{2m}
+ d \|S_r\|_2 h^{2m+2}
+d\sum_{j= r+1}^d \|\bphi_j\|_1^2 \, \lambda_j 
 \nonumber  \\
&+& 
d_4 \Big( \|S_R\|_2 \Delta t^2 +  \|S_R\|_2 h^{2m+2} + \sum_{j=  R+1}^d \|\bphi_j\|_1^2 \, \lambda_j \Big)
\Bigg),
\label{eq:proof5}
\end{eqnarray}
where $C^* = e^{d_1 \Delta t \sNn \lno \nabla \unp\rno ^4 }$. 

By dropping the third term on the LHS of \eqref{eq:proof5} and using \eqref{eq:error_eta_no_sum}, \eqref{eq:error_geta}, and the triangle inequality, we get 
\begin{eqnarray}
&& \lno {\bf u}^M - {\bf u}_r^M \rno^2
+ \nu
\Delta t \, \sum \limits_{k = 0}^{M-1} \lno \nabla \lp {\bf u}^{k+1} - {\bf u}_{r}^{k+1}\rp \rno^2
\nonumber \\
&\leq& C
\Bigg( \lp 1+\|S_r\|_2+\|S_R\|_2 \rp \Delta t^2 
+ h^{2m}
+ \lp 1+\|S_r\|_2+\|S_R\|_2 \rp h^{2m+2}
\\ \nonumber
&&+\sum_{j= r+1}^d \lp 1+\|\bphi_j\|_1^2 \rp \, \lambda_j 
 + \sum_{j=  R+1}^d \|\bphi_j\|_1^2 \, \lambda_j
\Bigg).
\end{eqnarray}
This completes the proof. 
\qed
 
\end{proof}

%% file: Numerics.tex
\section{Numerical Experiments}\label{num}
%
%
%
The goal of this section is twofold.
In Section \ref{ss_numerical_results_3D}, we investigate the physical accuracy of the $P_R$-VMS-POD-ROM.
To this end, we test the model in the numerical simulation of a 3D flow past a circular cylinder at $\mbox{Re} = 1000$.
The $P_R$-VMS-POD-ROM is compared with the POD-G-ROM and the ML-POD-ROM in which a constant EV is employed \cite{Aubry1988,wang2011closure}.
All the numerical results are benchmarked against DNS data.
A parallel CFD solver is employed to generate the DNS data \cite{Akhtar2008C}.   
For details on the numerical discretization, the reader is referred to the appendix in \cite{wang2011two}. 
To assess the physical accuracy of the the POD-ROMs, two criteria are employed:
(i) the time evolution of the POD coefficients, which measures the instantaneous behavior of the models; and 
(ii) the energy spectrum, which measures the average behavior of the models.
In Section \ref{ss_numerical_results_2D}, we illustrate numerically the theoretical error estimates in Theorem \ref{theorem_error}.
Specifically, we investigate the error's asymptotic behavior with respect to the time step, $\Delta t$, and the POD contribution to the error introduced by the EV term, $\sum \limits_{j=R+1}^{d} \|\bphi_j\|_1^2\lambda_j$.


\subsection{Physical Accuracy}
\label{ss_numerical_results_3D}

In this section, we test the $P_R$-VMS-POD-ROM in the numerical simulation of a 3D flow past a circular cylinder at $\mbox{Re} = 1000$.
By using the method of snapshots \cite{Sir87a}, we compute the POD basis $\{ {\boldsymbol \varphi}_1, \cdots, {\boldsymbol \varphi}_d \}$ from 1000 snapshots of the velocity field over 15 shedding cycles, i.e., $t\in [0, 75]$ (see Figure \ref{fig_3d_dns}). 
\begin{figure}[h!]
\centering
\includegraphics[width=0.4\textwidth]{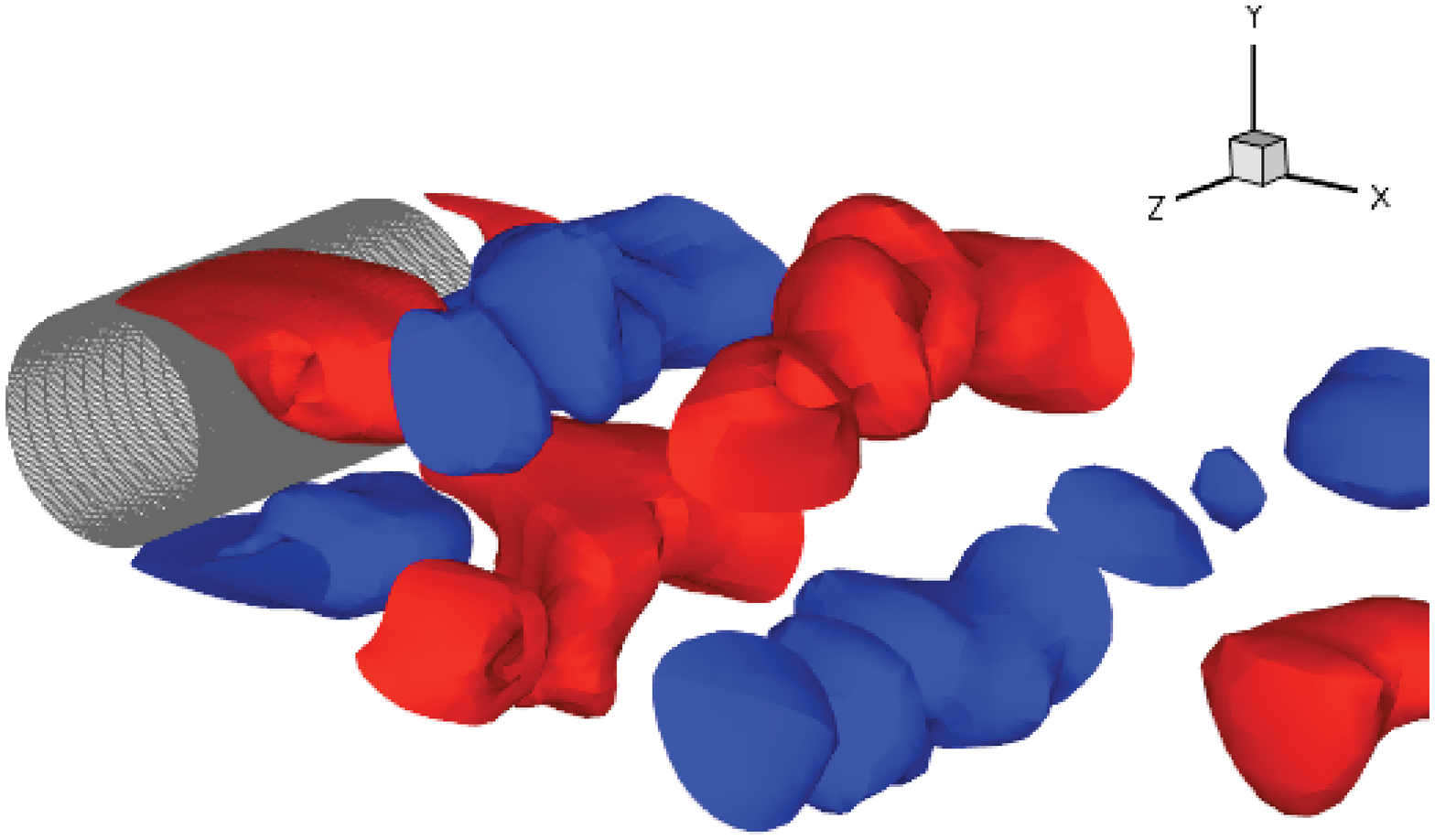}\hspace*{0.5cm}
\includegraphics[width=0.4\textwidth]{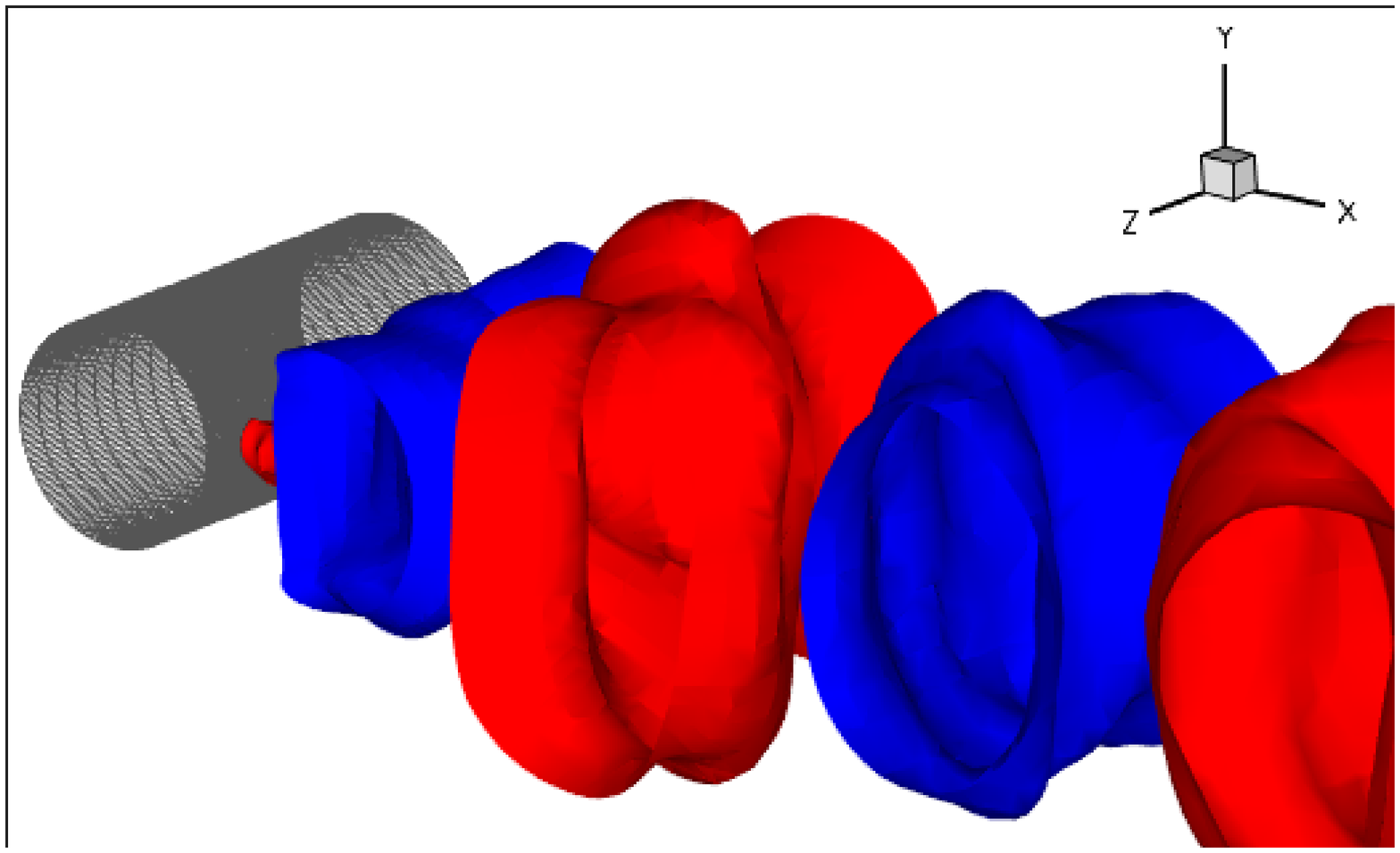}\vspace*{0.5cm}
\includegraphics[width=0.4\textwidth]{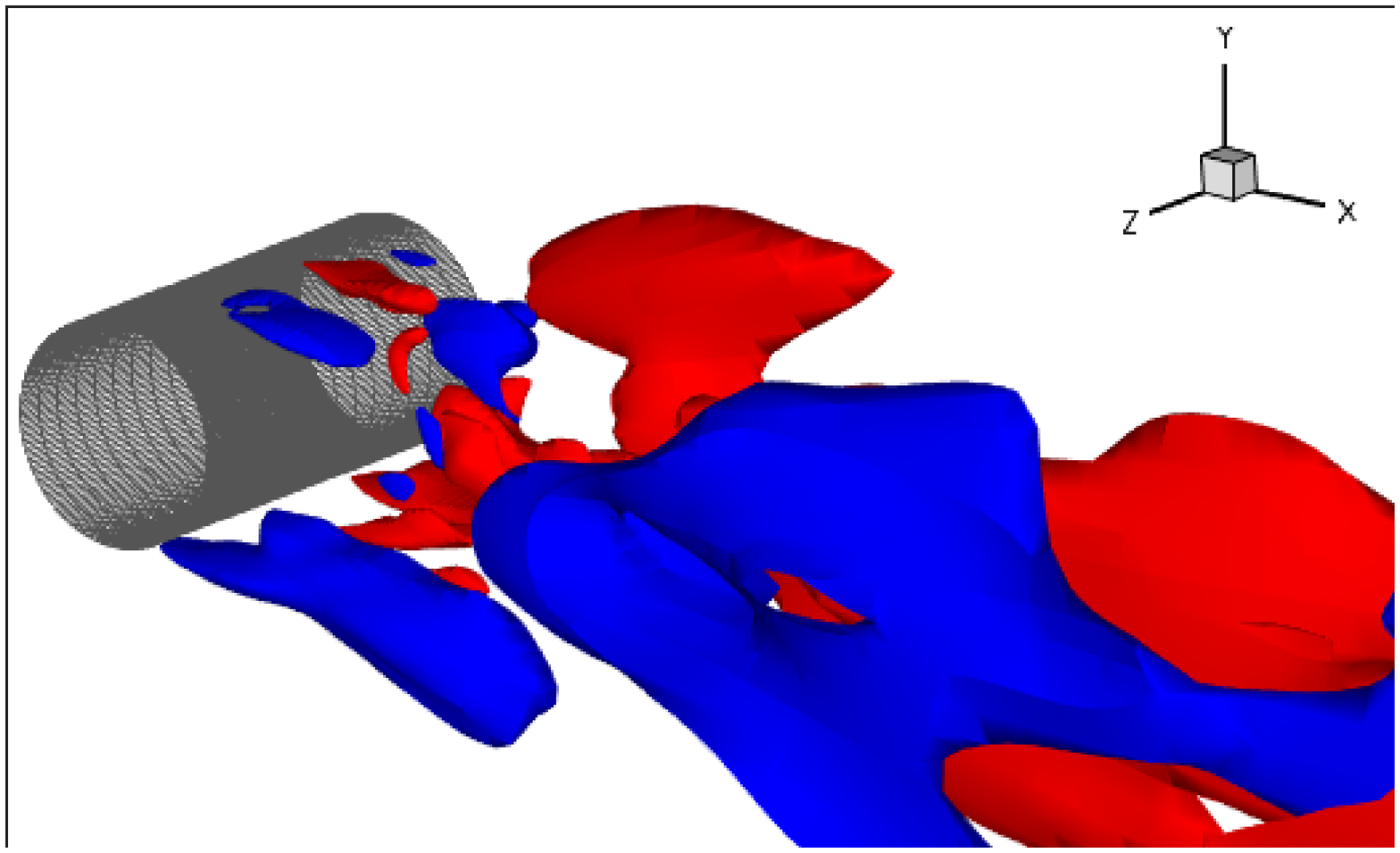}\hspace*{0.5cm}
\includegraphics[width=0.4\textwidth]{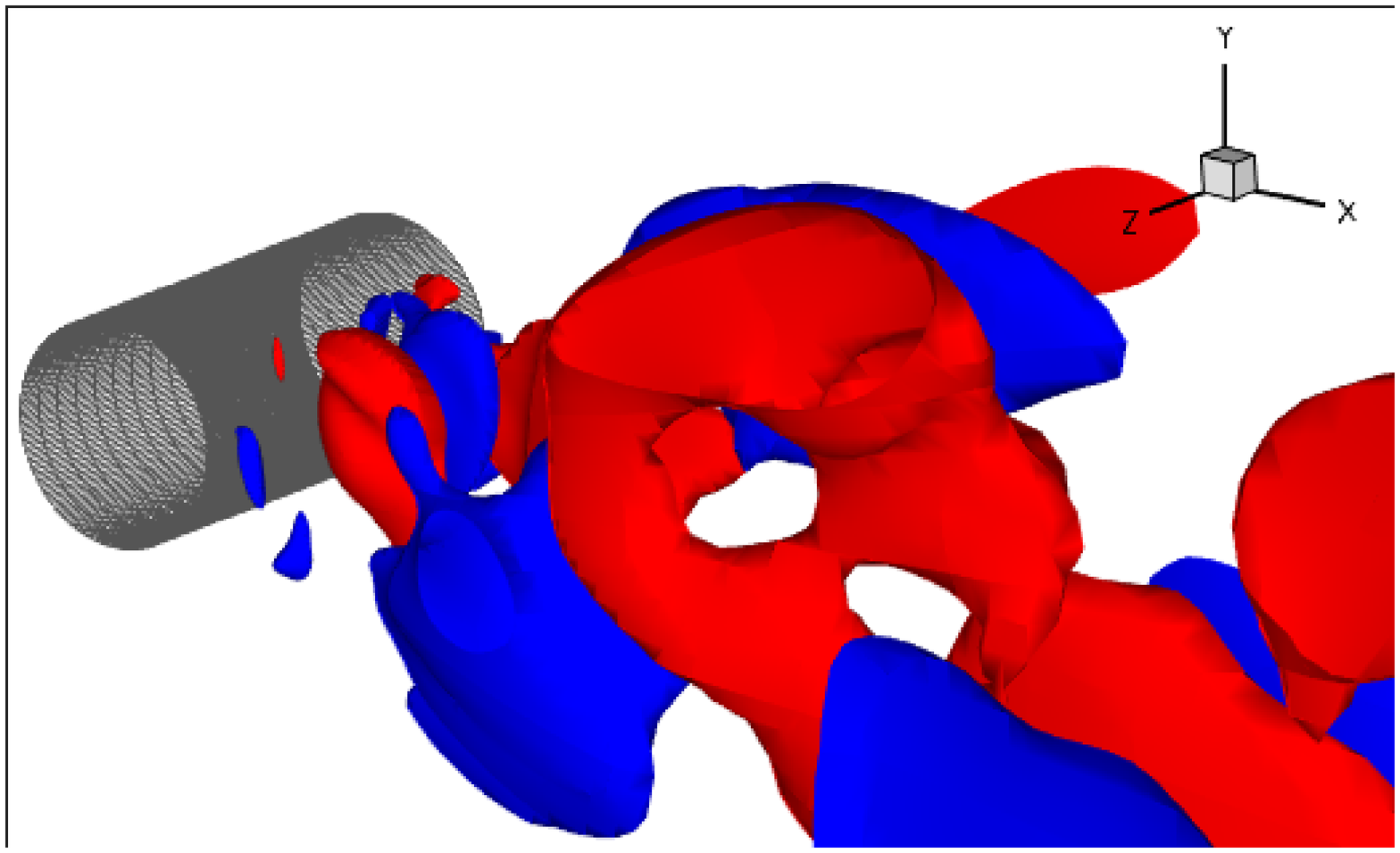}
\caption{
3D flow past a cylinder at $\mbox{Re}=1000$.
First streamwise POD mode (top left), 
first normal POD mode (top right),
third streamwise POD mode (bottom left), and
third normal POD mode (bottom right).
\label{fig_3d_dns}}
\end{figure}
These POD modes are then interpolated onto a structured quadratic FE mesh with nodes coinciding with the nodes used in the original DNS finite volume discretization. Six POD basis functions (r=6) are then used in all POD-ROMs that we investigate next.
These first six POD modes capture $84\%$ of the system's energy.
For all these POD-ROMs, the time discretization was effected by using the backward Euler method with $\Delta t= 7.5\times 10^{-3} $. 
We emphasize that the time interval used in the simulations of POD-ROMs is {\it four times larger} than that in which snapshots are generated, i.e., $t\in [0, 300]$. 
Thus, the {\em predictive} capabilities of the POD-ROMs are investigated. 
In Figure \ref{fig_3d_evo}, the time evolutions of the POD coefficients $a_1$ and $a_4$ are plotted.
The other POD coefficients have a similar qualitative behavior, so, for clarity, they are not included in our plots.
To determine the EV constants in the ML-POD-ROM and the $P_R$-VMS-POD-ROM, we run the models on the short time interval $[0,15]$ with several different values for the EV constants and choose the value that yields the results that are closest to the DNS results. 
This approach yields the following values for the EV constants: 
$\alpha= 3 \times 10^{-3}$ for the ML-POD-ROM \eqref{eq:ml-pod} and 
$\alpha=3.5\times 10^{-3}$ for the $P_R$-VMS-POD-ROM \eqref{eq:vms-pod} when $R=1$.
We emphasize that these EV constant values are optimal only on the short time interval tested, 
and they might actually be nonoptimal on the entire time interval $[0, 300]$ on which the POD-ROMs are tested.
Thus, this heuristic procedure ensures some fairness in the numerical comparison of the three POD-ROMs.

\begin{figure}[h!]
\centering
\begin{minipage}[c]{.02\linewidth}
{\small (a)}
\end{minipage}
\hspace*{0.05cm}
\begin{minipage}[c]{.47\linewidth}
\includegraphics[width=1\textwidth]{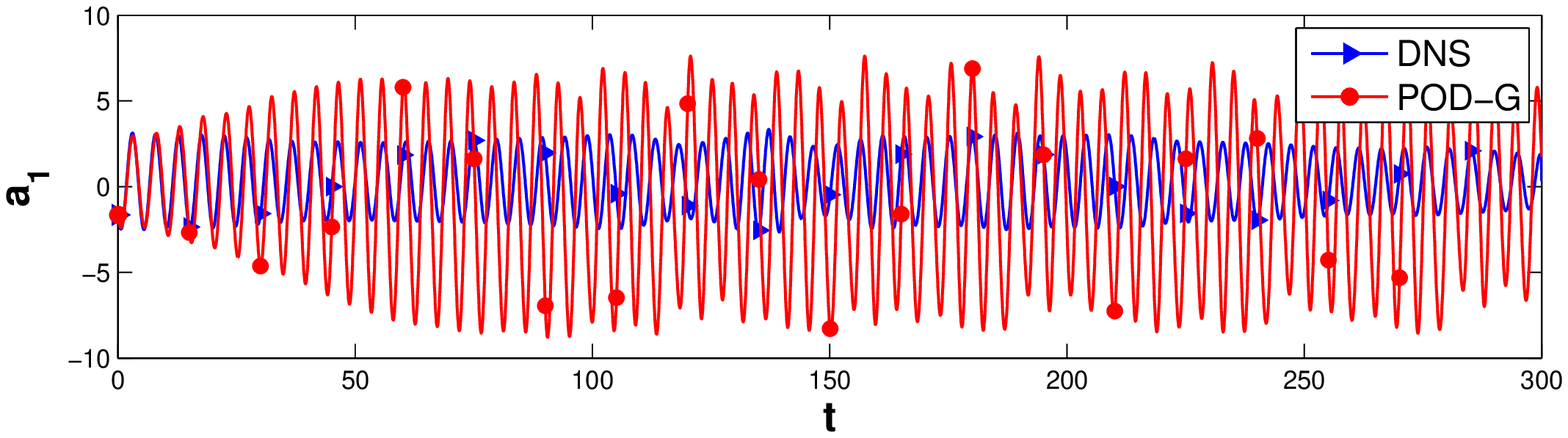}
\end{minipage}
\hspace*{0.05cm}
\begin{minipage}[c]{.47\linewidth}
\includegraphics[width=1\textwidth]{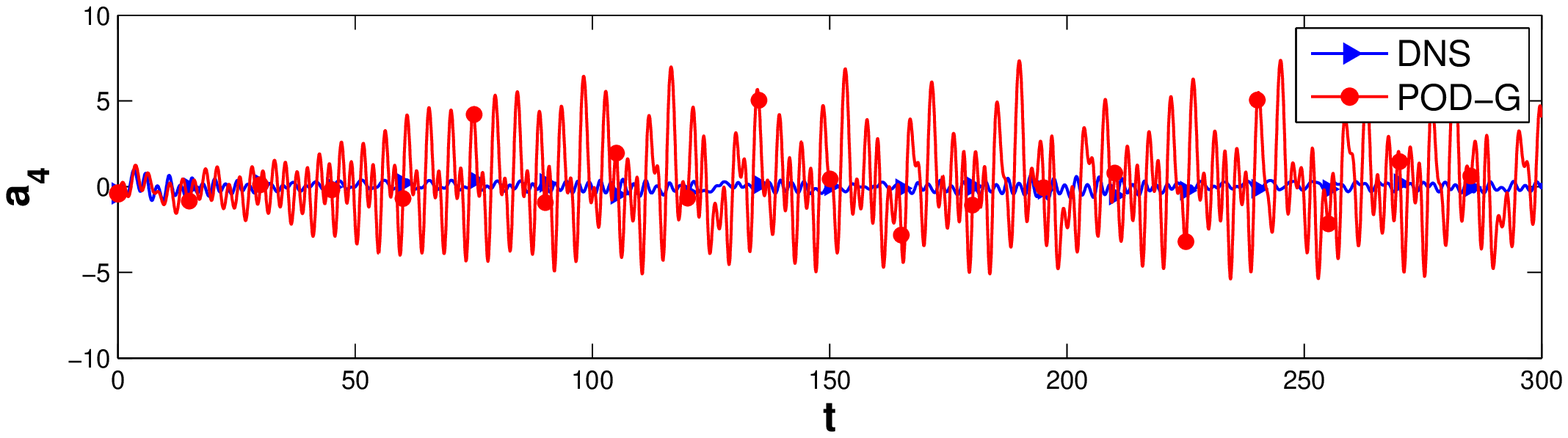}
\end{minipage}
\centering
\begin{minipage}[c]{.02\linewidth}
{\small (b)}
\end{minipage}
\hspace*{0.05cm}
\begin{minipage}[c]{.47\linewidth}
\includegraphics[width=1\textwidth]{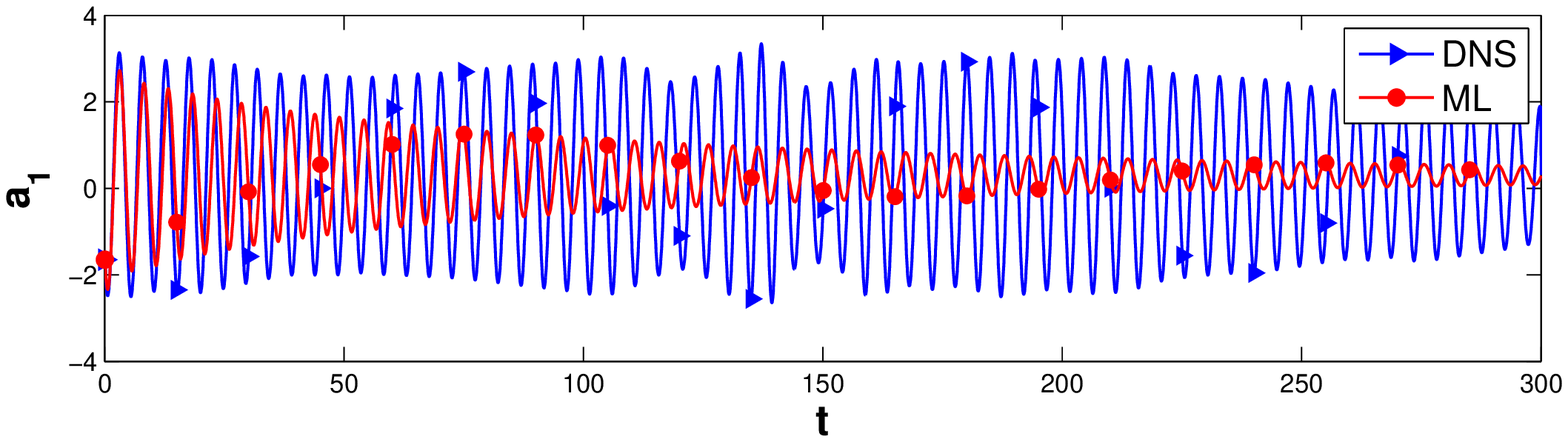}
\end{minipage}
\hspace*{0.05cm}
\begin{minipage}[c]{.47\linewidth}
\includegraphics[width=1\textwidth]{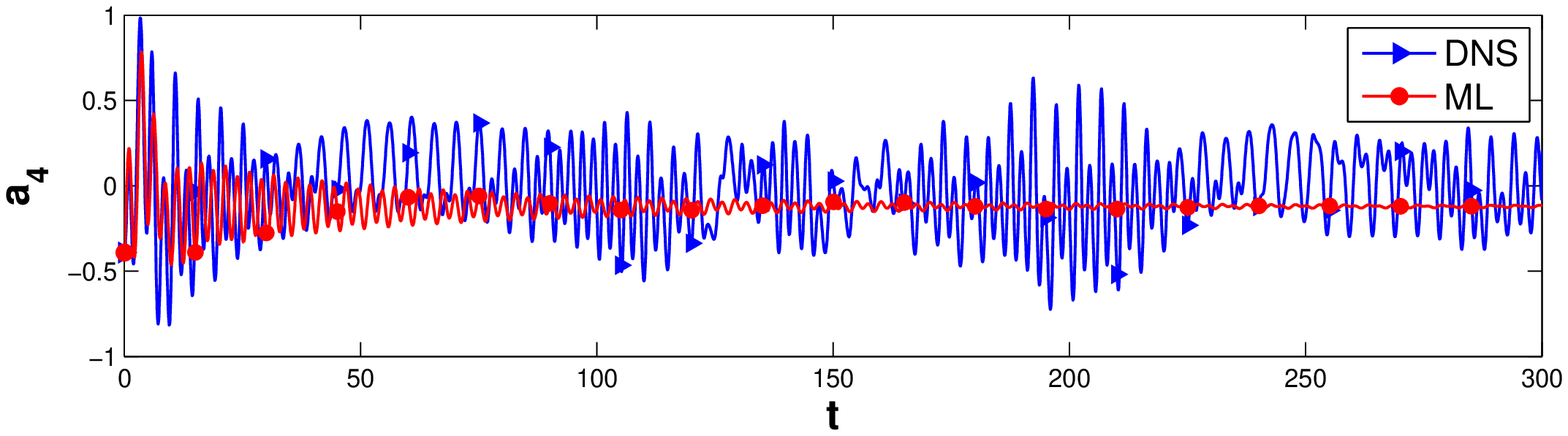}
\end{minipage}
\centering
\begin{minipage}[c]{.02\linewidth}
{\small (c)}
\end{minipage}
\hspace*{0.05cm}
\begin{minipage}[c]{.47\linewidth}
\includegraphics[width=1\textwidth]{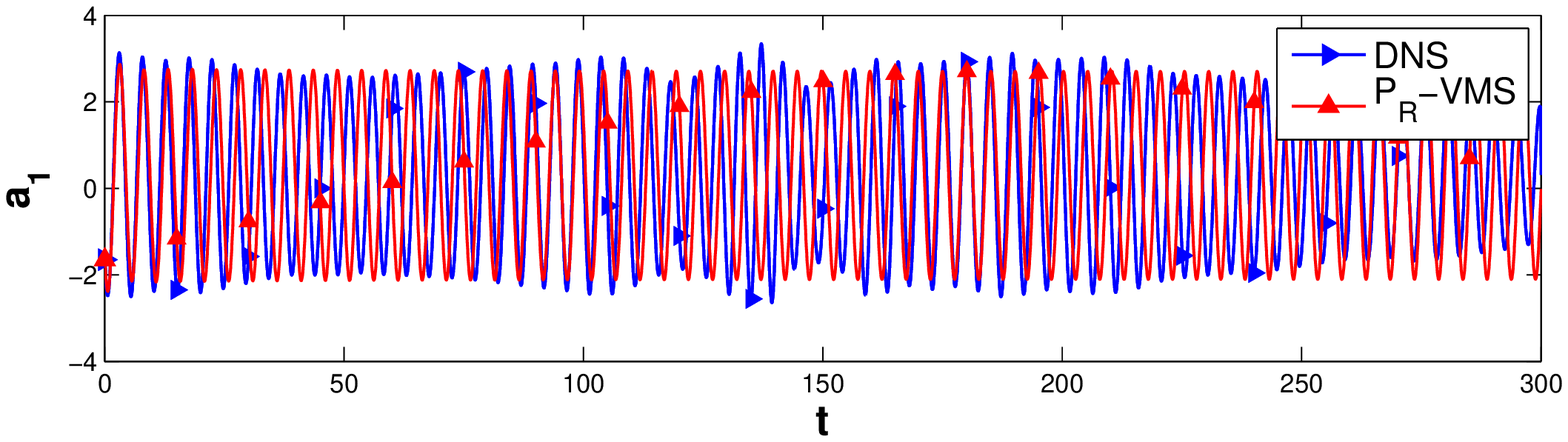}
\end{minipage}
\hspace*{0.05cm}
\begin{minipage}[c]{.47\linewidth}
\includegraphics[width=1\textwidth]{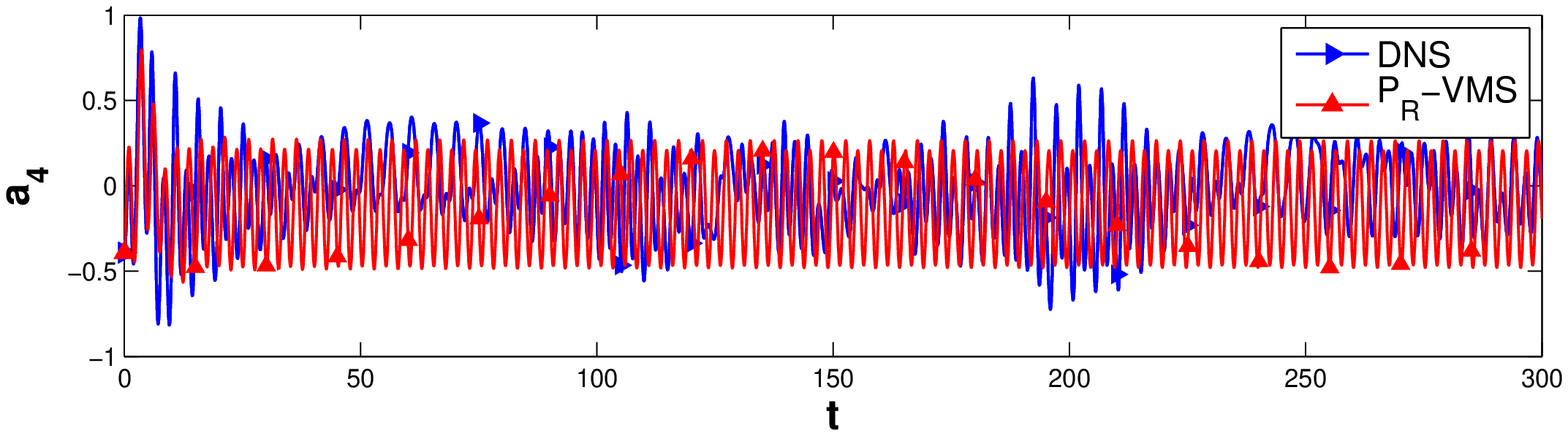}
\end{minipage}
\caption{
3D flow past a cylinder at $\mbox{Re}=1000$.
Temporal evolution of POD coefficients $a_1(\cdot)$ and $a_4(\cdot)$ over the time interval $[0, 300]$ for POD-G-ROM (a), ML-POD-ROM (b), and $P_R$-VMS-POD-ROM (c). 
\label{fig_3d_evo}
}
\end{figure}

\begin{figure}[h!]
\centering
\begin{minipage}[c]{.02\linewidth}
{\small (a)}
\end{minipage}
\hspace*{0.05cm}
\begin{minipage}[c]{.45\linewidth}
\includegraphics[width=1\textwidth]{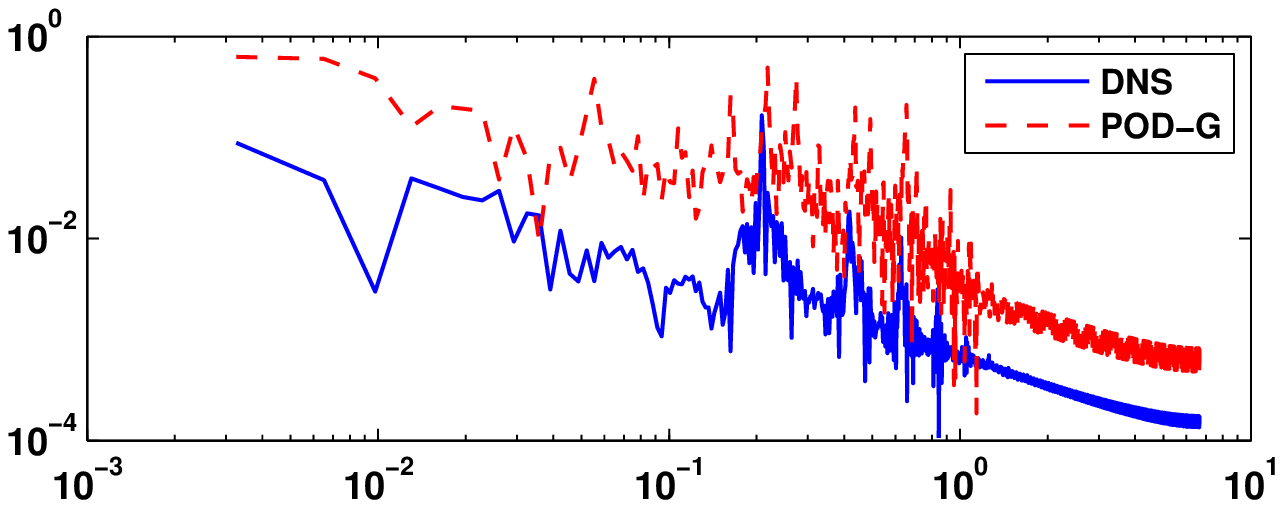}
\end{minipage}
\hspace*{0.05cm}
\begin{minipage}[c]{.02\linewidth}
{\small (b)}
\end{minipage}
\hspace*{0.05cm}
\begin{minipage}[c]{.45\linewidth}
\includegraphics[width=1\textwidth]{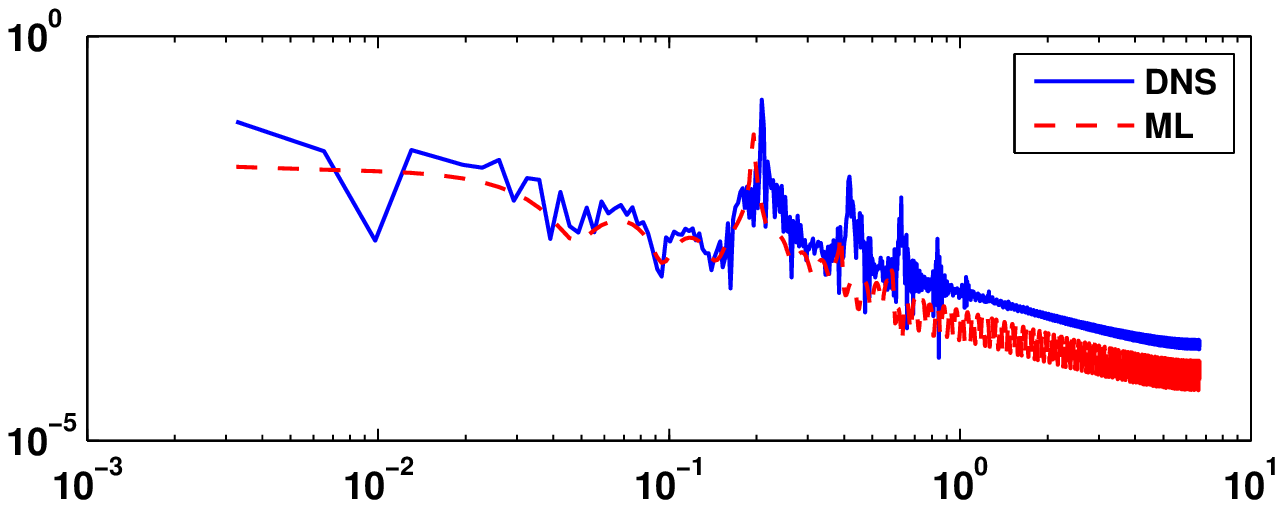}
\end{minipage}
\begin{minipage}[c]{.02\linewidth}
{\small (c)}
\end{minipage}
\hspace*{0.05cm}
\begin{minipage}[c]{.45\linewidth}
\includegraphics[width=1\textwidth]{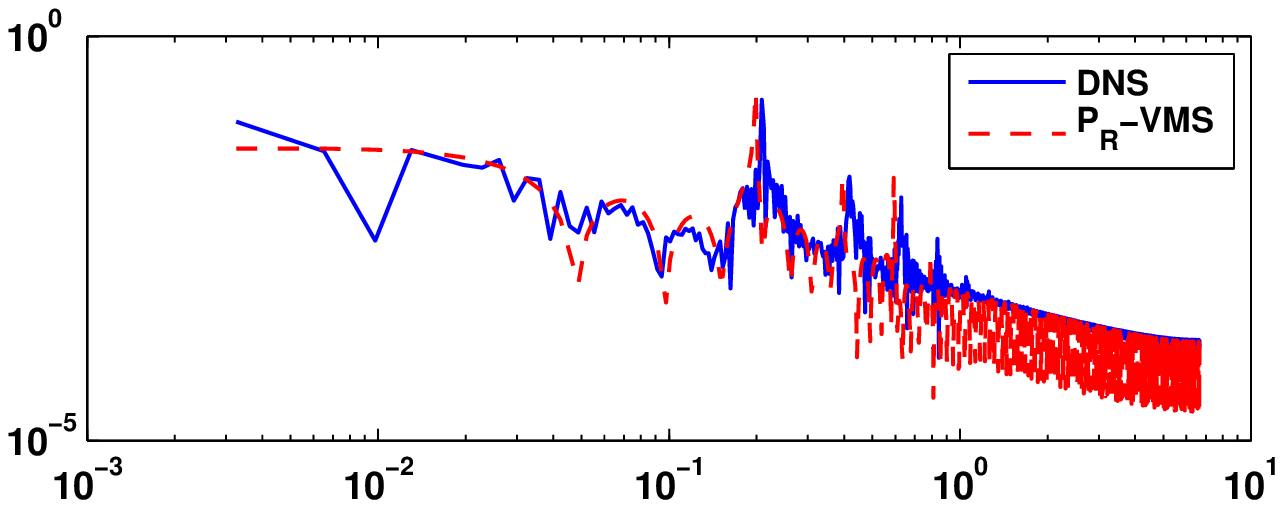}
\end{minipage}
\caption{
3D flow past a cylinder at $\mbox{Re}=1000$.
Comparison of energy spectrum of DNS with that of POD-G-ROM (a), ML-POD-ROM (b), and $P_R$-VMS-POD-ROM (c).
\label{fig_3d_egy}
}
\end{figure}

The POD-G-ROM \eqref{eq:pod-g} performs poorly, 
although it is computationally efficient (its CPU time is $118 \, s$). 
Indeed, the 
amplitude of the temporal evolution of the POD coefficient $a_{4}(\cdot)$ 
is {\it nine times larger} than that for the DNS projection.
The ML-POD-ROM's time evolutions of the POD coefficients $a_1$ and $a_4$ are also inaccurate. 
Specifically, although the time evolution at the beginning of the simulation (where the EV constant $\alpha$ 
was chosen) is relatively accurate, the accuracy significantly degrades toward the end of the simulation.
For example, as shown in Figure~\ref{fig_3d_evo}(b), the magnitude of $a_4$ at the end of the simulation is only one eighth of that of the DNS. 
The $P_R$-VMS-POD-ROM yields more accurate time evolutions than both the POD-G-ROM and the ML-POD-ROM for both $a_1$ and $a_4$, 
as shown in Figure~\ref{fig_3d_evo}(c). 
The $P_R$-VMS-POD-ROM is as efficient as POD-G-ROM, its CPU time being 129\, s.

Figure~\ref{fig_3d_egy} presents the energy spectra of the three POD-ROMs.
The three energy spectra are compared with the DNS energy spectrum.
For the energy spectra, we use the approach in \cite{wang2011closure} and we calculate the average kinetic energy of the nodes in the cube with side $0.1$ 
centered at the probe $(0.9992, 0.3575, 1.0625)$.
The energy spectrum of the POD-G-ROM \eqref{eq:pod-g} overestimates the energy spectrum of the DNS.
The energy spectrum of the ML-POD-ROM \eqref{eq:ml-pod}, on the other hand, underestimates the the energy spectrum 
of the DNS, especially at the higher frequencies.
Although it displays high oscillations at the higher frequencies, 
the $P_R$-VMS-POD-ROM \eqref{eq:vms-pod} has a more accurate spectrum than both the POD-G-ROM \eqref{eq:pod-g} and the ML-POD-ROM \eqref{eq:ml-pod}. 

%
%

\subsection{Numerical Accuracy}
\label{ss_numerical_results_2D}

In this section, we test the $P_R$-VMS-POD-ROM in the numerical simulation of the 2D incompressible NSE \eqref{eq:nse}. 
The exact velocity, $\bu = (u,v)$, has components $u = \frac{2}{\pi} \arctan(-500(y-t) ) \sin(\pi y)$, $v = \frac{2}{\pi}\arctan(-500(x-t) ) \sin(\pi x)$, and the exact pressure is given by $p = 0$.
The inverse of the Reynolds number is $\nu=10^{-3}$ and the forcing term is chosen to match the exact solution.
Taylor-Hood FEs are used to discretize the spatial domain $[0, 1]\times [0, 1]$. 
We collect snapshots over the time interval $[0, 1]$ at every $\Delta T = 10^{-2}$ by recording the exact values of $u$ and $v$ on the FE mesh with the mesh size $h=1/64$. 
After applying the method of snapshots, we obtain a POD basis set with the dimension of 101. 

In POD-ROMs, the backward Euler method is employed for time integration over the same time interval. To verify the numerical error of the $P_R$-VMS-POD-ROM \eqref{eq:vms-pod} with respect to the time step, $\Delta t$, we choose $h=1/64$, $r=99$, $R=95$ and $\alpha = 10^{-3}$. 
With this choice, $h^{2m}$ is on the order of $10^{-8}$, and $\sum\limits_{j = r+1}^{d}\|\bphi_j\|_1^2 \lambda_j$ and $ \sum\limits_{j = R+1}^{d}\|\bphi_j\|_1^2 \lambda_j$ are on the order of $10^{-4}$. 
Thus, asymptotically, the time discretization error dominates the total error in the theoretical error estimate \eqref{eq:theorem_error_1}. 
The total error in the $L^2$ norm at the final time, $\me = \| \bu^M - \bu^M_r \|$, is listed in Table \ref{table:dt} for decreasing values of the time step, $\Delta t$. 
A linear regression between $\me$ and $\Delta t$ (see Figure \ref{fig:dt}) shows 
\linelabel{rate1}
that the rate of convergence of the numerical error is nearly linear with respect to the time step, as predicted by the theoretical error estimate \eqref{eq:theorem_error_1}: 

\begin{equation*}
\me = 2.66 (\Delta t)^{0.96}.
\end{equation*}

\begin{figure}[htb]
\begin{minipage}[h]{0.45\linewidth}
\centering
\tabcaption{
The $P_R$-VMS-POD-ROM with $h=1/64$, $r=99$, $R=95$, and $\alpha = 10^{-3}$.
The total error in the $L^2$-norm at the final time, $\| \bu^M - \bu^M_r \|$, for decreasing values of the time step, $\Delta t$.
}\label{table:dt}
\centering
\begin{tabular}{|c|c|}
\hline
$\Delta t$&$\| \bu^M - \bu^M_r \|$\\
\hline
\multirow{1}{*}{$5\times10^{-3}$}      & $2.31\times 10^{-2}$\\
\multirow{1}{*}{$2.5\times10^{-3}$}    & $6.30\times 10^{-3}$\\
\multirow{1}{*}{$1.25\times10^{-3}$}  & $3.46\times 10^{-3}$ \\
\multirow{1}{*}{$6.25\times10^{-4}$}   & $2.05\times 10^{-3}$\\
\multirow{1}{*}{$3.13\times10^{-4}$}   & $1.45\times 10^{-3}$\\
\hline
\end{tabular}
\end{minipage}
\hspace{1cm}
\begin{minipage}[h]{0.5\linewidth}
\centering
\includegraphics[width=1\textwidth]{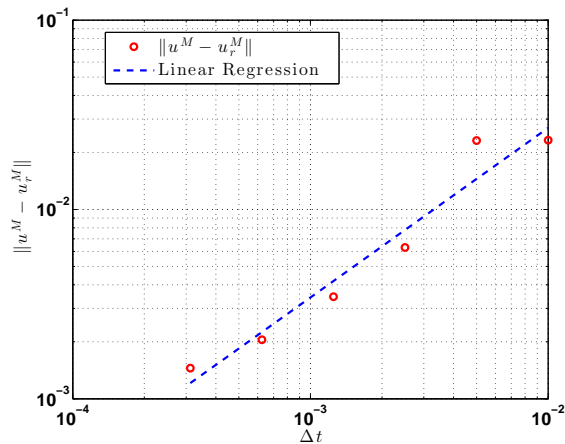}
\figcaption{
The $P_R$-VMS-POD-ROM with $h=1/64$, $r=99$, $R=95$ and $\alpha = 10^{-3}$.
A linear regression between the total error in the $L^2$-norm at the final time, $\me$, and the time step, $\Delta t$, is nearly linear : $\me \sim \mathcal{O}\lp (\Delta t)^{0.96}\rp$.
}\label{fig:dt}
\end{minipage}
\end{figure} 

To verify the numerical error of the $P_R$-VMS-POD-ROM with respect to $R$, we choose $h=1/64$, $\Delta t = 10^{-4}$, and $r=99$.
With this choice, $h^{2m}$ and $\Delta t^2$ are on the order of $10^{-8}$ and  $\sum\limits_{j = r+1}^{d}\|\bphi_j\|_1^2 \lambda_j$ is on the order of $10^{-4}$. 
Thus, asymptotically, the POD contribution to the error introduced by the EV term, $\sum_{j=  R+1}^d \|\bphi_j\|_1^2 \, \lambda_j $, dominates the total error in the theoretical error estimate \eqref{eq:theorem_error_1}. 
For a fixed $\alpha = 10^{-3}$, total error in the $L^2$ norm at the final time, $\me^2 = \| \bu^M - \bu^M_r \|^2$, is listed in Table \ref{table:R} for increasing values of $R$.
A linear regression between $\me^2$ and $\sum_{j=  R+1}^d \|\bphi_j\|_1^2 \, \lambda_j$ (see Figure \ref{fig:R}) shows 
\linelabel{rate2}
an almost quadratic rate of convergence, which is higher than the linear rate of convergence predicted by the theoretical error estimate \eqref{eq:theorem_error_1}: 

\begin{equation*}
\me^2 = c \left(\sum_{j=  R+1}^d \|\bphi_j\|_1^2 \, \lambda_j \right)^{1.94},
\end{equation*}
where $c=3.2\times 10^{-9}$. 


\begin{figure}[htb]
\begin{minipage}[h]{0.45\linewidth}
\centering
\tabcaption{
The $P_R$-VMS-POD-ROM with $h=1/64$, $\Delta t = 10^{-4}$, $r=99$, and $\alpha = 10^{-3}$.
The square of the total error in the $L^2$-norm at the final time, $\| \bu^M - \bu^M_r \|^2$, for increasing values of $R$.
}\label{table:R}
\begin{tabular}{|c|c|c|}
\hline
                      R & $\sum_{j=  R+1}^d \|\bphi_j\|_1^2 \, \lambda_j$ & $\| \bu^M - \bu^M_r \|^2$ \\
\hline
\multirow{1}{*}{6}   & $2.18\times 10^{2}$ & $1.43\times 10^{-4}$ \\
\multirow{1}{*}{10} & $1.99\times 10^{2}$ & $1.05\times 10^{-4}$ \\
\multirow{1}{*}{16} & $1.73\times 10^{2}$ & $6.70\times 10^{-5}$ \\
\multirow{1}{*}{24} & $1.43\times 10^{2}$ & $4.04\times 10^{-5}$ \\
\multirow{1}{*}{34} & $1.10\times 10^{2}$ & $2.39\times 10^{-5}$ \\
\multirow{1}{*}{45} & $7.80\times 10^{1}$ & $1.54\times 10^{-5}$ \\
\multirow{1}{*}{56} & $5.37\times 10^{1}$ & $8.92\times 10^{-6}$ \\
\hline
\end{tabular}
\end{minipage}
\hspace{1cm}
\begin{minipage}[h]{0.5\linewidth}
\centering
\includegraphics[width=1\textwidth]{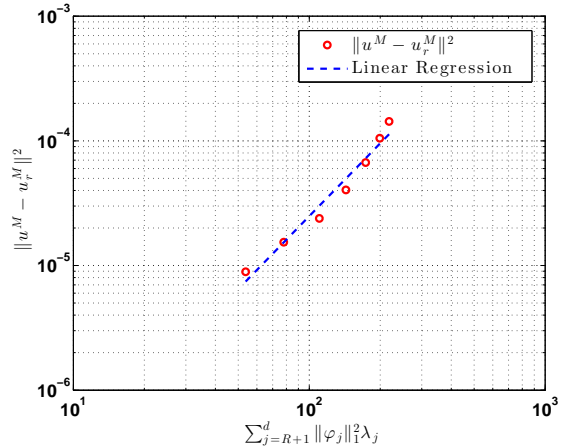}
\figcaption{
The $P_R$-VMS-POD-ROM with $h=1/64$, $\Delta t = 10^{-4}$, $r=99$, and $\alpha = 10^{-3}$.
A linear regression between the square of the total error in the $L^2$-norm at the final time, $\me^2$, and the the POD contribution to the error introduced by the EV term, $\sum_{j=  R+1}^d \|\bphi_j\|_1^2 \, \lambda_j$: $\me^2 \sim \mathcal{O}\lp \left(\sum_{j=  R+1}^d \|\bphi_j\|_1^2 \, \lambda_j \right)^{1.94} \rp$.
}\label{fig:R}
\end{minipage}
\end{figure}


%% file: Conclusion.tex
\section{Conclusions}\label{con}

In this paper, we proposed a new ROM for the numerical simulation of turbulent incompressible fluid flows. 
This model, denoted $P_R$-VMS-POD-ROM, utilizes a VMS method and a projection operator to model the effect of the high index POD modes that are not included in the ROM.
A rigorous error estimate was derived for the full discretization of the $P_R$-VMS-POD-ROM.
All the contributions to the total error were considered: the spatial discretization error (due to the FE discretization), the temporal discretization error (due to the backward Euler method), and the POD truncation error.
The $P_R$-VMS-POD-ROM was also tested in the numerical simulation of a 3D flow past a circular cylinder at $Re=1000$.
The numerical tests showed that the $P_R$-VMS-POD-ROM is both physically accurate and computationally efficient.
Furthermore, the numerical results illustrated the theoretical error estimates.

We note that the EV coefficient $\alpha$ in the $P_R$-VMS-POD-ROM is simply chosen to be a constant.
We plan to extend this theoretical and numerical study by considering more accurate choices for the EV coefficients, such as the Smagorinsky model \cite{wang2011closure,ullmann2010pod}.
We also plan to investigate this model in more complex physical settings, such as the Boussinesq equations \cite{ravindran2010error}. 
Finally, we plan to reduce the computational cost of the $P_R$-VMS-POD-ROM by a different treatment of the time discretization of the VMS term.
